\documentclass[11pt,oneside,reqno]{article}

\usepackage{amsmath,amsthm,amssymb,color,bm,graphicx, float}

\usepackage{array}
\usepackage[margin=2.5cm,a4paper]{geometry}
\usepackage{bbm}
\usepackage{paralist}
\usepackage{mathrsfs}
\usepackage[breaklinks=true]{hyperref}

\usepackage[square]{natbib} 
\usepackage{bm}
\renewcommand{\cite}{\citep*}

\numberwithin{equation}{section}
\allowdisplaybreaks[4]

\theoremstyle{plain}
\newtheorem{theorem}{Theorem}[section]
\newtheorem{proposition}[theorem]{Proposition}
\newtheorem{lemma}[theorem]{Lemma}
\newtheorem{corollary}[theorem]{Corollary}

\theoremstyle{definition}
\newtheorem{definition}[theorem]{Definition}

\makeatletter

\renewcommand{\phi}{\varphi}

\newcommand{\eq}{\eqref}

\newcommand{\bigo}{\mathrm{O}}

\def\E{\mathbbm{E}}
\newcommand{\Z}{\mathbb{Z}}
\newcommand{\R}{\mathbb{R}}

\newcounter{ctr}\loop\stepcounter{ctr}\edef\X{\@Alph\c@ctr}%
	\expandafter\edef\csname s\X\endcsname{\noexpand\mathscr{\X}}
	\expandafter\edef\csname c\X\endcsname{\noexpand\mathcal{\X}}
	\expandafter\edef\csname b\X\endcsname{\noexpand\boldsymbol{\X}}
	\expandafter\edef\csname I\X\endcsname{\noexpand\mathbbm{\X}}
	\expandafter\edef\csname r\X\endcsname{\noexpand\mathrm{\X}}
\ifnum\thectr<26\repeat

\def\ba#1{\begin{align*}#1\end{align*}}
\def\ban#1{\begin{align}#1\end{align}}

\def\given{\typeout{Command 'given' should only be used within bracket command}}
\newcounter{@bracketlevel}
\def\@bracketfactory#1#2#3#4#5#6{
\expandafter\def\csname#1\endcsname##1{%
\addtocounter{@bracketlevel}{1}%
\global\expandafter\let\csname @middummy\alph{@bracketlevel}\endcsname\given%
\global\def\given{\mskip#5\csname#4\endcsname\vert\mskip#6}\csname#4l\endcsname#2##1\csname#4r\endcsname#3%
\global\expandafter\let\expandafter\given\csname @middummy\alph{@bracketlevel}\endcsname
\addtocounter{@bracketlevel}{-1}}%
}
\def\bracketfactory#1#2#3{%
\@bracketfactory{#1}{#2}{#3}{relax}{1mu plus 0.25mu minus 0.25mu}{0.6mu plus 0.15mu minus 0.15mu}
\@bracketfactory{b#1}{#2}{#3}{big}{1mu plus 0.25mu minus 0.25mu}{0.6mu plus 0.15mu minus 0.15mu}
\@bracketfactory{bb#1}{#2}{#3}{Big}{2.4mu plus 0.8mu minus 0.8mu}{1.8mu plus 0.6mu minus 0.6mu}
\@bracketfactory{bbb#1}{#2}{#3}{bigg}{3.2mu plus 1mu minus 1mu}{2.4mu plus 0.75mu minus 0.75mu}
\@bracketfactory{bbbb#1}{#2}{#3}{Bigg}{4mu plus 1mu minus 1mu}{3mu plus 0.75mu minus 0.75mu}
}
\bracketfactory{clc}{\lbrace}{\rbrace}
\bracketfactory{clr}{(}{)}
\bracketfactory{cls}{[}{]}

\def\norm#1{\Vert#1\Vert}

\def\abs#1{\vert#1\vert}

\def\bbbabs#1{\biggl\vert#1\biggr\vert}

\renewcommand\section{\@startsection {section}{1}{\z@}%
{-3.5ex \@plus -1ex \@minus -.2ex}%
{1.3ex \@plus.2ex}%
{\center\small\sc\mathversion{bold}\MakeUppercase}}

\def\subsection#1{\@startsection {subsection}{2}{0pt}%
{-3.5ex \@plus -1ex \@minus -.2ex}%
{1ex \@plus.2ex}%
{\bf\mathversion{bold}}{#1}}

\def\subsubsection#1{\@startsection{subsubsection}{3}{0pt}%
{\medskipamount}%
{-10pt}%
{\normalsize\itshape}{\kern-2.2ex. #1.}}

\def\blfootnote{\xdef\@thefnmark{}\@footnotetext}

\makeatother

\def\s#1{^{(#1)}}

\newcommand\Dir{{\rm{Dir}}}

\newcommand\ta{{\bm{a}}}
\newcommand\tb{{\bm{b}}}
\newcommand\tX{{\bm{X}}}
\newcommand\tZ{{\bm{Z}}}
\newcommand\tz{{\bm{z}}}

\newcommand\tx{{\bm{x}}}

\newcommand\te{{\bm{e}}}

\newcommand\ti{{\bm{i}}}

\newcommand\ty{{\bm{y}}}

\newcommand\tV{{\bm{V}}}

\newcommand\tu{{\bm{u}}}
\newcommand\tU{{\bm{U}}}
\newcommand\tk{{\bm{k}}}
\newcommand\tj{{\bm{j}}}

\def\ds{\bar \nabla^{K}}

\def\d{\delta}
\def\Bin{\text{Bin}}

\def\s{\nabla^{K}}

\newcommand{\bs}{\bar \nabla^{K}}

\begin{document}

\title{\sc\bf\large\MakeUppercase{Steady-state Dirichlet approximation of the Wright-Fisher model using the prelimit generator comparison approach of Stein's method}}
\author{\sc Anton Braverman and Han~L.~Gan}
\date{\it Northwestern University and University of Waikato}
\maketitle 
\begin{abstract}
The Wright-Fisher model, originating in~\cite{Wright31} is one of the canonical probabilistic models used in mathematical population genetics to study how genetic type frequencies evolve in time. In this paper we bound the rate of convergence of the stationary distribution for a finite population Wright-Fisher Markov chain with parent independent mutation to the Dirichlet distribution. Our result improves the rate of convergence established in~\cite{GRR17} from $\bigo(1/\sqrt{N})$ to $\bigo(1/N)$. The results are derived using Stein's method, in particular, the prelimit generator comparison method.
\end{abstract}


\section{Introduction}
We focus on approximating the stationary distribution for a finite Wright-Fisher Markov chain with parent independent mutation where the population has fixed size $N$ and fixed number of allele types $K$. 
We represent this model as a discrete time Markov chain $\tU(t)$ in the space 
\begin{align*}
\s = \left\{\tu \in \delta \IZ^{K-1} : u_i \geq 0, \sum_{i=1}^{K-1}u_i \leq 1\right\},
\end{align*}
where $\delta = 1/N$ is a scaling parameter, $U_i(t)$ denotes the fraction of genes that are of type $1 \leq i \leq K-1$, and $1 - \sum_{i=1}^{K-1}U_i(t)$ is the fraction of genes of type $K$. For any $\tu, \ty \in \s$ and probabilities $p_1, \ldots, p_K$ such that $\sum_{i=1}^K p_i \leq 1$, the transition probabilities $P_{\tu}(\ty) = \Pr(U(1) = \ty | U(0) = \tu)$ of this process  satisfy 
\ban{
P_{\tu}(\ty) =&\ {N\choose N y_1, \ldots, N y_{K-1}} \prod_{j=1}^{K-1} \Big( u_{j} \big(1 - \sum_{i=1}^{K} p_{i}\big) + p_{j}   \Big)^{N y_j}.\label{eq:tp}
}
Define   $\ds = \{\tu \in \IR^{K-1} : x \geq 0, \sum_{i=1}^{K-1}x_i \leq 1\}$ and let $\tilde P_{y}: \ds \to \IR$ be the natural extension of $P_{\ty}(\tu)$ to $\ds$. 
Our Markov chain is irreducible, aperiodic, and positive recurrent because its state space is finite, and we let $\tU$ denote the vector having the unique stationary distribution.

As it will be useful later, we informally give a common interpretation of the model and how it models changes in allele types over time. In each generation of fixed size $N$, given the parent generation, each individual in the offspring population independently chooses a parent, uniformly at random. In addition to this random genealogy structure, a random mutation structure is added such that each individual offspring independently has a probability of mutating to type $i$ with probability $p_i$. Otherwise, with probability $1 - \sum_{i=1}^K p_i$, the offspring does not mutate and takes on the type of their parent. Note that in this structure each child could mutate to the same type as their respective parent.

In~\cite{GRR17}, bounds for quantities of the form $\abs{\E h(\tU) - \E h(\tZ)}$ are derived, where $\tZ$ is  an appropriately chosen Dirichlet random variable, and $h$ is any general test function with two bounded derivatives and bounded Lipschitz constant on the second derivative. Under the typical assumption for these models that the mutation probabilities are rare, in the sense that $p_i = \bigo(1/N)$, \cite{GRR17} establish an upper bound on $\abs{\E h(\tU) - \E h(\tZ)}$ that is of order $\bigo(1 /\sqrt N)$. It has been anecdotally conjectured to the authors that this bound may not be of the optimal order, and that the correct order may be of order $\bigo (1/N)$. In particular, ~\cite{EN77} derived a bound of order $\bigo (1/N)$, but their result requires bounded sixth derivatives and is restricted to the case where $K=2$ (beta distribution). In this paper, we derive a bound of order $\bigo (1/N)$ for approximating the stationary distribution assuming four bounded derivatives. Before we present the main result, we first define the Dirichlet distribution and our approximating metric.

We define the Dirichlet distribution with parameters $\bm\beta = (\beta_1, \ldots, \beta_K)$, where $\beta_1 > 0, \ldots, \beta_K > 0$, to be supported on the $(K-1)$-dimensional open simplex  
\ba{
\text{int$(\ds)$} = \left\{ \tx = (x_1, \ldots, x_{K-1}): x_1 > 0, \ldots, x_{K-1} > 0, \sum_{i=1}^{K-1}x_i < 1 \right\} \subset \mathbb{R}^{K-1}.
}
The Dirichlet distribution has density
\ban{
\psi_\beta(x_1, \ldots x_{K-1}) = \frac{ \Gamma(s)}{\prod_{i=1}^K \Gamma(\beta_i)} \prod_{i=1}^K x_i^{\beta_i - 1}, \quad \tx \in \text{int$(\ds)$}, \label{eq:densitydirichlet}
}
where $s = \sum_{i=1}^K \beta_i$, and we set $x_K = 1 - \sum_{i=1}^{K-1} x_i$. We assume that our mutation probabilities $p_{i}$ satisfy 
\begin{align}
p_{i} = \frac{\beta_{i}}{2N}, \quad 1 \leq i \leq K, \label{ass:main}
\end{align}
for some fixed $\bm\beta$ and all $N > 0$.
%

The metric we will be using is the Lipschitz type metric defined as follows. For any vector $\ta \in \Z^{K-1}$ consisting of non-negative integer values and a function $f: \R^{K-1} \to \R$, let 
\begin{align}
D^{\ta} f(\tx) = \frac{\partial^{a_{K-1}}}{\partial x_{K-1}^{a_{K-1}}} \cdots \frac{\partial^{a_{1}}}{\partial x_{1}^{a_{1}}} f(\tx), \quad \tx \in \R^{K-1}, \label{eq:Da}
\end{align}
and define 
\begin{align}
&\mathcal{M}_{j} =  \Big\{h: \R^{K-1}  \to \R,\  \sup_{\tx}\abs{ D^{\ta} h(\tx) } \leq 1,\ 1 \leq \norm{\ta}_{1} \leq j,\ \ta \geq 0  \Big\}. \label{def:Mj}
\end{align}  
Then for any random vectors $\tV,\tV' \in \R^{K-1}$, set
\begin{align*}
d_{\mathcal{M}_{j}}(\tV,\tV') =  \sup_{h \in \mathcal{M}_{j}} \big| \E h(\tV) - \E h(\tV') \big|.
\end{align*}
Lemma 2.2 of \cite{GM16}  proves that $\mathcal{M}_{3}$ is a convergence-determining class; i.e., $d_{\mathcal{M}_{3}}(\tV,\tV') \to 0$ implies $\tV$ and $\tV'$ converge in distribution.  Their result can be readily extended to show that $\mathcal{M}_{4}$, the class of functions used in this paper,  is also convergence determining.  
\begin{theorem}
\label{thm:main}
Let $\tU$ denote the random vector with stationary distribution for the transition function~\eq{eq:tp}, assume that the mutation probabilities satisfy \eqref{ass:main} for some $\bm\beta>0$, and let $\tZ$ be a Dirichlet random variable with parameter vector $\bm\beta$. Then for all $N > 0$,  there exists a constant $C(\bm\beta, K)$ that is independent of $N$, such that 
\begin{align}
d_{\mathcal{M}_{4}}(\tU ,\tZ)  \leq C(\bm\beta,K) \Big( \frac{1}{N} +   \frac{1}{N^{\beta_{K}/2}} ( 1 + N^{2}) \Big). \label{eq:main}
\end{align}
\end{theorem}
A few comments are in order. Although our proof allows us to recover the explicit constant $C(\bm\beta,K)$, keeping track of it is impractical as it quickly becomes very messy.  
Furthermore, when $\beta_{K}/2 \geq 6$, then the bound in \eqref{eq:main} is $\bigo(1/N)$. We believe that the $ ( 1 + N^{2})/N^{\beta_{K}/2}$ term in the upper bound is merely an artifact of our methodology. It appears because we use an interpolation operator is based on forward differences and, as a result, we have to treat the case when $\tZ$ is close to the ``right'' boundary of $\bs$ separately. We expect that our methodology could be refined to get rid of this term. Doing so would require modifying our interpolation operator to use forward and backward differences when close to the ``left'' and ``right'' boundaries of $\bs$, respectively,  and central differences ``in the middle'' to smoothly transition between the forward and backward differences. This undertaking  is beyond the scope of this paper.
 
The primary tool used in this paper to prove the main results is Stein's method. Stein's method is a powerful tool in probability theory that is used to derive an explicit bound for the difference between two probability distributions. Typically one aims to use it to find an upper bound on the errors incurred when approximating an intractable target distribution with a commonly used simple reference distribution. It was first developed for the Normal distribution in~\cite{stein72} to bound the approximation errors when applying the central limit theorem, and it has since been developed numerous distributions, such as Poisson~\cite{Chen75,BHJ}, beta~\cite{GR13, Dobler2015}, Dirichlet~\cite{GRR17}, Poisson-Dirichlet~\cite{GR21}, negative binomial~\cite{BP99}, exponential~\cite{FulmanRoss2013} to just name a few. For many more examples and applications, see for example the surveys or monographs~\cite{Ross11,Chatterjee2014,introstein,LRS2017}. In the following we give a brief introduction to Stein's method.

To successfully apply Stein's method, one of the main approaches is what is known as the generator method, first pioneered in~\cite{B88}. Below we give a brief description of Stein's method in general, with a particular focus on the generator method, and details of our approach. In this brief description we discuss the univariate case, but note that the multivariate case is analogous. Our goal is to bound the difference between the typically unknown or intractable law of our target random variable $X$ with the law of a well understood and simple reference random variable $Z$. Stein's method can usually be summarised in the following three main steps.
\begin{enumerate}
\item Identify a characterising operator $\mathcal{G}_Z$ or identity that is satisfied only by the distribution of the reference random variable $Z$. In the generator method, the characterising operator is a generator of a Markov chain or diffusion process, and the reference distribution is the associated stationary distribution. The generator characterises its associated stationary distribution through the identity that $\E \mathcal{G}_Zf(Z) = 0$ for all suitable functions $f$ if and only if $Z$ follows the stationary distribution. 
\item For any arbitrary function $h$, solve for the function $f_h$ that satisfies
\ban{
\mathcal{G}_Z f_h(x) = h(x) - \E h(Z).
}
Then by setting $x = X$ and taking expectations,
\ban{
\abs { \E \mathcal{G}_Z f_h(X)} = \abs{ \E h(X) - \E h(Z)}.\label{eq:stneqa}
}
Properties of the function $f_h$ turn out to be crucial to derive a good bound with Stein's method. Typically one will require good bounds on $f_h$ and its derivatives or differences if $Z$ is continuous or discrete. Using the generator method, one can usually express $f_h$ in terms of the semi-group of the process defined by $\mathcal{G}_Z$, and exploit properties and couplings of the process to to found such bounds.  
\item The goal is to bound~\eq{eq:stneqa} for all $h$ from a rich enough family of test functions $\mathcal H$, where $\mathcal H$ is typically a convergence determining class. Rather than directly bounding $\abs{ \E h(W) - \E h(Z)}$, which would typically require knowledge of the unknown density or distribution function of $X$, the final step is to derive a bound for $\abs { \E \mathcal{G}_Z f_h(X)}$ which is more tractable. Standard approaches often involve Taylor expansions and couplings.
\end{enumerate}
In~\cite{GRR17}, the above approach is used where $\mathcal{G}_\tZ$ is the generator of the Wright-Fisher diffusion,
\ban{
\mathcal{G}_\tZ f(\tz) := \frac{1}{2}\left[\sum_{i,j=1}^{K-1} z_i(\delta_{i j}-z_j) \frac{\partial f}{\partial z_i \partial z_j}(\tz)
+\sum_{i=1}^{K-1}(\beta_i-s z_i)\frac{\partial f}{\partial z_i}(\tz)\right], \quad \tz \in \text{int}(\ds), \label{eq:gz}
}
where $s = \sum_{i=1}^K \beta_i$ and $\delta_{ij}$ denotes the Kronecker delta function. The stationary distribution associated to this generator is the Dirichlet distribution on $\text{int$(\ds)$}$ with parameters $\beta_1, \ldots, \beta_K$. Recalling the definition of $\tU$ as the stationary distribution associated with~\eq{eq:tp}, letting $\tZ \sim \Dir(\beta_1, \ldots, \beta_K)$, a bound for $\abs{\E h(\tU) - \E h (\tZ)}$ is derived by finding a bound for $\abs{ \E \mathcal{G}_Z f_h(\tU)}$.  The following (Stein) lemma formalises the link between $\cG_\tZ$ and $\tZ$.

\begin{lemma}
\label{lem:dirbar}
The random vector $\tZ \sim \Dir(\beta_1, \ldots, \beta_K)$  if and only if, for all $f \in C^{2}(\bs)$ with bounded partial derivatives up to the second order  and Lipschitz continuous second-order partial derivatives, 
\begin{align*}
\E \cG_{\tZ}  f(\tZ)  = 0.
\end{align*}
\end{lemma}

In~\cite{B22}, a variation of the generator method was innovated, namely the \emph{prelimit generator comparison approach}, which is the approach we use in this paper. We briefly describe the general idea of the approach and how it compares with the traditional generator method. The full details of our approach are contained in Section~2 and the appendix. The traditional generator comparison approach works by noting that~\eq{eq:stneqa} yields
\ba{
 \abs{ \E h(X) - \E h(Z)} = \abs { \E \mathcal{G}_Z f_h(X) - \E \mathcal{G}_X f_h(X)},
} 
as where $\mathcal{G}_X$ is also characterising operator/generator for $X$, and hence $\E \mathcal{G}_X f_h(X) = 0$. Typically $X$ will be a discrete object, and $Z$ will be its continuous limit. The generator comparison approach then follows the intuition that if $X$ is approximately equal to $Z$, then the operators $\mathcal{G}_X$ and $\mathcal{G}_Z$ should also be approximately equal, and the distributional distance between $X$ and $Z$ can be quantified by the differences in $\mathcal{G}_X$ and $\mathcal{G}_Z$. In a sense, this approach takes the discrete object $X$ and evaluates it with respect to the continuous operator $\mathcal G_Z$. The prelimit generator comparison approach swaps the roles of the continuous and discrete terms. 

Let $h: \delta \mathbb{Z}^d \mapsto \mathbb R$, where $\delta > 0$ be a test function defined on the lattice $\delta \mathbb{Z}^d$. For a random vector $\tU$ which takes values on $\delta \mathbb{Z}^d$, suppose there is a characterising operator $\mathcal G_\tU$ and given $h$ one can find the solution to the Stein equation
\ban{
\mathcal{G}_\tU f_h(\tu) = h(\tu) - \E h(\tU). \label{eq:stneqd}
}
Then, naively speaking at least, for some continuous $\tZ$ and its characterising operator $\mathcal{G}_\tZ$,
\ban{
\abs{\E h(\tZ) -\E h(\tU)} = \abs{\E \mathcal{G}_\tU f_h(\tZ) - \E \mathcal{G}_\tZ f_h(\tZ)}. \label{eq:1}
}
The general approach remains similar to the typical generator approach, the bound is reliant upon $\mathcal{G}_\tU$ and $\mathcal{G}_\tZ$ being close. One can now see however that in comparison to the standard generator comparison method, we are putting the continuous object $\tZ$ into the discrete (state-space) generator for $\tU$. This can be advantageous if the solution $f_h$ to the Stein equation~\eq{eq:stneqd} with respect to the discrete generator is tractable. Unfortunately the above equation~\eq{eq:1} is not that straightforward as $\mathcal{G}_\tU$ is only defined on the lattice $\delta \mathbb{Z}^d$, but we wish to input the continuous object $\tZ$. To address this issue, a smoothing interpolation operator for $\mathcal{G}_\tU$ is required, which will be described in detail in this paper.

Although the task of  interpolating $\mathcal{G}_\tU$ is conceptually straightforward, its execution is nontrivial. Thus, examples of the prelimit approach in practice are crucial for its proliferation.  To date, only two examples exist. A simple single-server queue ($M/M/1$ model) is considered in \cite{Brav2022}, while \cite{Brav2023} considers a much more involved queueing model --- a load-balancing model under the  join-the-shortest-queue policy. The former is a simple one-dimensional model that does not fully illustrate the challenges of  interpolating $\mathcal{G}_\tU$, while the latter example is extremely involved due to the complicated dynamics of the queueing model. The Wright-Fisher model we consider   falls nicely between these two examples in terms of difficulty --- our model is multi-dimensional, highlighting all the challenges of interpolating $\mathcal{G}_\tU$, whilst the task of bounding the Stein factors, which is unrelated to interpolation, is relatively straightforward. Additionally, this paper  refines the original implementation of the prelimit generator approach in \cite{Brav2022}. We present several results in Appendix~\ref{app:interpolator} that simplify working with the interpolating operator. Specifically, we state and prove  Lemma~\ref{lem:lipschitz}, Lemma~\ref{lem:productrule}, Proposition~\ref{prop:interperror}, and Corollary~\ref{cor:polynomials}. We anticipate that these results will help future users of the approach.

The approach used in~\cite{GRR17}, follows the traditional generator method described earlier, whereas we use the prelimit generator comparison method in this paper. There are advantages and disadvantages to either approach. If we consider the main three steps in applying Stein's method discussed earlier, steps 1 and 3 are more or less the same in both approaches. Step 1 is in the traditional approach involves finding a diffusion operator that characterises the Dirichlet distribution. In the prelimit approach, in addition to the same operator, we also require an operator for the discrete population Wright-Fisher Markov chain stationary distribution, which is not difficult. Step 3 in the traditional approach in~\cite{GRR17} uses an exchangeable pair coupling, and ultimately the main work involves a series of moment calculations for the discrete stationary distribution. In the prelimit approach, near identical calculations are required. The primary difference between the two approaches lies in step 2. In~\cite{GRR17}, solving the Stein equation and bounding the derivatives of the solution, known as the Stein factors, is a lengthy process and requires knowledge of coalescent theory and a dual process representation of the Wright-Fisher diffusion process governed by~\eq{eq:gz}. These bounds are one of the primary contributions of~\cite{GRR17}. In contrast, using the prelimit approach, we instead require Stein factors for the Markov chain associated with~\eq{eq:tp}, which we bound in Lemma~\ref{lem:steinfactors}. The bounds are simple, elegant and require only a short proof using elementary Markov chain knowledge and couplings. The price we pay to use this simpler approach for the Stein factors, is the requirement for an interpolation operator, which leads to numerous technical difficulties, and this is the main trade off between the two approaches.

The remainder of the paper will be as follows. In Section~2, we outline the proof to Theorem~\ref{thm:main}, then provide a number of technical lemmas, and given these lemmas, we prove Theorem~\ref{thm:main}. The third and final section serves as an appendix which includes the proofs of the technical lemmas.
\section{Proof of the main theorem}
\subsection{Notation}
\begin{itemize}
\item For any function $f : \nabla^K \to \IR$, and a non-negative integer valued vector $\ta$, let
\ban{
B_{i} (f)  = \sup \{\abs{\Delta^{\ta} f(\tu)} : \norm{\ta}_{1}=i, \tu \in \s, \text{ and } \tu + \delta \ta \in \s \}, \label{eq:Bi}
}
where $\Delta^\ta$ refers to the forwards difference operator with step size $\delta$, and analogous to the definition of~\eq{eq:Da}, $\ta$ indicates in what directions the forward differences are taken. Note that given $\norm{\ta}_{1}=i$, $\Delta^\ta$ is the composition of $i$ forward differences in the directions indicated by the entries of $\ta$ and not a single forward difference in the direction of $\ta$ treated as a whole. For example, if $\ta = (1,0, \ldots, 0)$, then $\Delta^{\ta} f(\tu) = f(\tu + \delta\te_1) - f(\tu)$ and if $\ta = (1,1,0,\ldots 0)$ then $\Delta^\ta f(\tu) = f(\tu + \delta\te_1 + \delta\te_2) - f(\tu + \delta\te_1) - f(\tu + \delta\te_2) + f(\delta\tu)$, where $\te_i$ denotes the usual standard basis vector with $1$ in the $i$-th component.
\item We reserve the variables $\tu$ (and $\tU$) to emphasise when a function is defined on the lattice $\nabla^K$, and $\tx$ (and $\tX$) when the function is defined on the continuous simplex $\bar \nabla^K$.
\item The vector $\te$ is reserved to denote a K-dimensional vector of ones, that is $\te = (1,1,\ldots,1)$. Furthermore, any inequalities with respect to $\te$ are intended to be element by element wise. That is if $\tx \leq \te$, then $x_i \leq e_i$ for all $i$.
\item We use $\Sigma = \sum_{i=1}^K p_i$ to denote the sum of the mutation probabilities in~\eq{eq:tp}. 
\end{itemize}
\subsection{Outline of the proof}
Recall that our goal is to bound $d_{\mathcal M_4}(\tU, \tZ) = \sup_{h \in \mathcal M_4} \abs{\E h(\tU) - \E h(\tZ)}$. We achieve this bound in three main steps.

\begin{enumerate}
\item \textbf{Solve the Stein equation:} Identify a charaterising operator for $\cG_\tU$ for $\tU$ and then for any function $h \in \mathcal{M}_4$, solve for $f_h$ that satisfies the Stein equation
\ban{
\cG_\tU f_h(\tu) = h(\tu) - \E h(\tU). \label{step1}
}
\item \textbf{The interpolation operator:} We would like to simply substitute $\tu = \tZ$ in~\eq{step1} and take expectations, but $\cG_\tU$ is not well defined for continuous objects as it characterises $\tU$ which is discrete. We therefore extend $f_h(\tu)$ to take arguments from $\bar \nabla^K$ using an carefully chosen interpolation operator $A$ which satisfies $Af_h(\tx) = f_h(\tu)$ for all $\tx  = \tu \in \nabla^K$ and $A$ applied to a constant equals that constant. Then by applying $A$ again to~\eq{step1},
\ban{
A(\cG_\tU(A f_h))(\tx) = A h(\tx) - \E h(\tU). \label{step2}
}
\item \textbf{Generator expansion:} Noting that $\E \cG_Z A f_h(\tZ) = 0$, then by setting $\tx = \tZ$ in~\eq{step2} and  taking expectations,
\ban{
\E[ A \cG_\tU A f_h(\tZ) - \cG_\tZ A f_h(\tZ)] = \E A h(\tZ) - \E h(\tU).\label{step3}
}
We therefore need to carefully bound the left hand side. We achieve this via Taylor expansion of $A \cG_\tU A f_h(\tx)$. The choice of $A$ plays a crucial role here as we will have specifically chosen $A$ in such a manner that the derivatives of $A f_h$ correspond to the finite differences of $f_h$ up to the fourth order.
\end{enumerate}

\subsubsection{Solving the Stein equation}
As the concept of a generator for a discrete time Markov chain is relatively uncommon, for the benefit of readers we spend some time to define the generator and the general form of the solution to its Stein equation.
\begin{definition}
Let $\tU(t)$ be the Wright-Fisher Markov chain with parent independent mutation, rescaled to take values on $\s \subset \delta \mathbb{Z}^d$. For any function $f$ from a suitable class of test functions $\mathcal F$, we define the \emph{Markov chain generator} of this process $\mathcal{G}_\tU$ as
\ba{
\mathcal{G}_\tU f(\tu) = \E[ f(\tU(1) )| \tU(0) = \tu] - \E f(\tu), \quad \tu \in \s.
}
\end{definition}
Note that we will use subscript notation on generator operators, for example $\mathcal{G}_\tU$, to associate a generator with its stationary distribution $\tU$.   
\begin{lemma}
Let $\tU(t)$ denote a Markov chain governed by the generator $\mathcal{G}_\tU$, then for all $h \in \mathcal{M}_4$, the function
\ban{
f_h(\tu) = \sum_{t=0}^\infty \left[\E (h (\tU(t)) | \tU(0) = \tu) - \E h(\tU)\right], \quad \tu \in \s,\label{eq:stnsol}
}
is well defined, and is the solution  to
\ba{
\mathcal{G}_\tU f_h(\tu) = h(\tu) - \E h(\tU), \quad \tu \in \s.
}
\end{lemma}
\begin{proof}
This can be shown by adapting Lemma~2 of~\cite{Brav2022} to the discrete-time setting (see also Lemma~1 of~\cite{B88}).
\end{proof}

\subsubsection{The interpolation operator}
\label{sec:interpolatorinbody}
Our proof relies on the ability to extend any function $f: \delta \Z^{d} \to \R$ to  $\R^{d}$  in a way that  the derivatives of the extension correspond to the finite differences of $f$. Many such extensions are possible, but we use an interpolating seventh-order Hermite polynomial spline. The spline is a linear operator $A$ acting on functions $f: \delta \Z^{d} \to \R$ and returning an extension $Af:\R^{d} \to \R$. When $d = 1$,  
\begin{align}
A f(x) = \sum_{i=0}^{4} \alpha^{k(x)}_{k(x)+i}(x)f(\delta(k(x)+i)), \quad x \in \R, \label{eq:Ad1}
\end{align}
where $k(x) = \lfloor x/\delta\rfloor$ and $\alpha_{k+i}^{k}: \R \to \R$ are weights defined for all $k \in \Z$ and $i = 0, \ldots, 4$, making $A f(x)$   a weighted sum of the five points $f(\delta k(x)), \ldots, f(\delta (k(x)+4))$. We use five points so that  the derivatives of $A f(x)$ capture the finite differences of $f(x)$  up to the fourth order.  

The details about $\alpha_{k+i}^{k}(x)$ and the definition of $A f(x)$ for $d >1$ are left to to Section~\ref{app:interpolator}.   For the purposes of this section, it suffices to know that $A$ is a linear operator, that $A f(\tx) = f(\tu)$ for $\tx = \tu \in \delta \Z^{d}$, that $Af(\tx)$ is twice continuously differentiable, and that $A$ applied to a constant equals that constant. 


\subsubsection{Generator expansion}
We first define a discrete analog of $\mathcal M_j$.
\begin{align*}
\mathcal{M}_{disc,j}(c) =  \Big\{h: \delta \Z^{K-1}  \to \R,\  \abs{\Delta^{\ta} h(\delta \tk)}  \leq c \delta^{\norm{\ta}_{1}},\ 1 \leq \norm{\ta}_{1} \leq j,\  \delta \tk \in \delta \Z^{K-1} \Big\}, \quad j \geq 1.
\end{align*}
The following lemma, which we prove in Section~\ref{app:genexpansion}, shows that the Taylor expansion of $A (\cG_{\tU} (A f_h))(\tx)$   equals $\cG_{\tZ} A f_h(\tx)$ plus an error term on a subset of $\bs$. To state it, we define 
\begin{align}
\s_{N} =&\ \{\tu \in \s  : u_{K} \geq 10K/\sqrt{N}   \} = \{\tu \in \s  : \sum_{i=1}^{K-1} u_i \leq 1 - 10K/\sqrt{N}   \},  \label{eq:ds4}
\end{align}
and let $\bs_{N} =  \text{Conv}(\s_{N})$. Note that $\s_{N} \neq \emptyset$ if $1 - 10K/\sqrt{N}  > 0$, or $N > 100K^2$, which we assume going forward. 
\begin{lemma}
\label{lem:agx}
The extension $A (\cG_{\tU} (A f_h))(\tx)$ is well defined for all $\tx \in \bs_{N}$. Furthermore, if the mutation probabilities satisfy \eqref{ass:main} for some $\bm\beta>0$,  then  
\begin{align*}
A (\cG_{\tU} (A f_h))(\tx) =&\ \delta \cG_{\tZ} A f_h(\tx)  + \epsilon_{\cG}(\tx), \quad \tx \in \bs_{N},
\end{align*}
where $\cG_{\tZ} A f_h(\tx)$ is defined in \eqref{eq:gz}, and, for all $N > 0$, 
\begin{align*}
\abs{\epsilon_{\cG}(\tx)} \leq C(\bm\beta,K) \big(\delta^{5} B_{1}(f_h)+ B_{2}(f_h) + \delta^{-1} B_{3}(f_h) + \delta^{-2} B_{4}(f_h) \big).
\end{align*} 
\end{lemma}
The final main ingredient needed to prove Theorem~\ref{thm:main} are bounds for the Stein factors $B_i(f_h)$.
\begin{lemma}\label{lem:steinfactors}
Let $f_h(\tu)$ be defined as in~\eq{eq:stnsol}, then for all $h \in \cM_{disc,4}(C)$, then
\ba{
B_i(f_h) \leq \frac{C\delta^i}{1- (1-\Sigma)^i}.
}
\end{lemma}
\subsection{Proof of Theorem~\ref{thm:main}}
We recall that to prove Theorem~\ref{thm:main} we need to bound $d_{\mathcal{M}_{4}}(U,Y)$. 
Recall that $\delta = 1/N$. The following lemma relates $\mathcal{M}_{j}$ to $\mathcal{M}_{disc,j}(c)$. We prove it in Section~\ref{app:interpolator} after stating Proposition~\ref{prop:interperror}.
\begin{lemma}
\label{lem:convdet}
There exist $C,C'(d)> 0$  such that for any $\tV \in \delta \Z^{d}$ and $\tV' \in \R^{d}$, 
\begin{align}
d_{\mathcal{M}_{4}}(\tV,\tV' ) \leq \sup_{\substack{h \in \mathcal{M}_{disc,4}(C)  }} \abs{\E h(\tV) - \E Ah(\tV')} + C'(d)\delta^{4}. \label{eq:mdisc}
\end{align}
\end{lemma}
Going forward, when we write $\mathcal{M}_{disc,j}(C)$, the constant $C$ is assumed to be the one in Lemma~\ref{lem:convdet}.  We require the following auxiliary lemma, which is proved in Section~\ref{app:interpolator}.
\begin{lemma}
\label{lem:secondterm}
Assume that the mutation probabilities satisfy \eqref{ass:main} for some $\bm\beta > 0$. There exist $C(K), C(\bm\beta,K) > 0$ such that for any $h \in \mathcal{M}_{disc,4}(C)$, 
\begin{align*}
&\abs{h(\tu)}, \abs{A h(\tZ)} \leq \abs{h(0)} + C(K), \quad \text{ and }  \quad \abs{\cG_{\tZ} A f_h(\tZ)} \leq C(\bm\beta,K)\delta^{-3}  B_{1}(f_h).
\end{align*}
\end{lemma}

We are now ready to prove the main theorem.
\begin{proof}[Proof of Theorem~\ref{thm:main}]
Fix $h \in  \mathcal{M}_{disc,4}(C)$.  As a consequence of Lemma~\ref{lem:convdet}, it suffices to bound $\abs{\E h(\tU) - \E A h(\tZ)}$ to prove Theorem~\ref{thm:main}. To bound this quantity, we recall the Stein equation \eqref{eq:stneqd}
\begin{align*}
\cG_{\tU} f_h(\tu) = h(\tu) - \E h(\tU), \quad \tu \in \s.
\end{align*}
For convenience, we extend $f_h(\tu)$   to $\delta \Z^{K-1}$ by setting  $f_h(\tu) = 0$ for $\tu \in \delta \Z^{K-1} \setminus \s$, so that $A f_h(x)$  can be defined for all $\tx \in \R^{K-1}$.   Since $A f_h(\tx) = f_h(\tu)$ for $\tx  = \tu \in \s$, the Stein equation is equivalent to 
\begin{align*}
\cG_{\tU} (Af_h)(\tu) = h(\tu) - \E h(\tU), \quad \tu \in \s.
\end{align*}
Note that $\cG_{\tU} (Af_h)(\tu)$ is only defined for $\tu \in \s$ even though $Af_h(\tx)$ is defined on $\R^{K-1}$.

Applying $A$ to $h(\tu) - \E h(\tX)$ and recalling from Lemma~\ref{lem:agx} that $A (\cG_{U} (A f_h))(\tx)$ is well defined for $\tx \in \bs_{N}$, we see that for any $\tx \in \R^{K-1}$, 
\begin{align*}
A h(\tx) - \E h(\tU) =&\ A \big( \E h(\tU) - h\big)(\tx) 1(\tx  \in \bs_{N}) + A \big( \E h(\tU) - h\big)(\tx) 1(\tx   \not \in \bs_{N}) \notag \\
=&\ A (\cG_{U} (A f_h))(\tx) 1(\tx   \in \bs_{N}) + \big( \E h(\tU) - A h(\tx)\big) 1(\tx   \not\in \bs_{N}).  
\end{align*} 
We claim that  $A f_h(\tx)$ satisfies the conditions of Lemma~\ref{lem:dirbar}, implying that $ \E \cG_{\tZ} A f_h(\tZ) = 0$. Our claim holds because   $A f_h(\tx)$ is twice continuously differentiable, and $Af_h(\tx)$ has compact support, which follows by  Theorem~\ref{thm:interpolant_def} of Section~\ref{app:interpolator} since $f_h(\tu) = 0$ for $\tu \in \delta \Z^{K-1} \setminus \s$.  Setting $\tx = \tZ$ and taking expected values  yields  
\begin{align*}
\E A h(\tZ) - \E h(\tU) =&\ \E \big( A (\cG_{\tU} (A f_h))(\tZ) - \delta \cG_{\tZ} A f_h(\tZ) \big) 1(\tZ \in \bs_{N})\\
&+ \E \big( \E h(\tU) - A h(\tZ)- \delta\cG_{\tZ} A f_h(\tZ)\big) 1(\tZ  \not \in \bs_{N}) \\
=&\ \E \epsilon_{\cG}(\tZ) 1(\tZ \in \bs_{N}) +  \E \big( \E h(\tU) - A h(\tZ)- \delta\cG_{\tZ} A f_h(\tZ)\big) 1(\tZ  \not \in \bs_{N}).
\end{align*}
To bound the first term, we combine Lemma~\ref{lem:agx} with the Stein factor bounds in Lemma~\ref{lem:steinfactors} and assumption~\eq{ass:main} implies $\Sigma = \bigo(1/N)$ to conclude that 
\begin{align*}
\abs{\epsilon_{\cG}(\tx)} &\leq C'(\bm\beta,K) \big(\delta^{5} B_{1}(f_h)+ B_{2}(f_h) + \delta^{-1} B_{3}(f_h) + \delta^{-2} B_{4}(f_h) \big) \leq C(\bm\beta,K) \frac{1}{N}.
\end{align*}
Let us bound the second term.  Recall that $\tZ$ has density given by \eqref{eq:densitydirichlet}, which implies that $Z_{K}\sim$Beta$(\beta_{K}, s - \beta_{K})$. Therefore, 
\begin{align}
 \IP(\tZ \not \in \bs_{N}) =&\ \IP( Z_K \leq 10K/\sqrt{N}) \notag \\
 =&\ \int_{0}^{10K/\sqrt{N}} \frac{\Gamma(s)}{\Gamma(s-\beta_{K}) \Gamma(\beta_{K})} (1-x_{K})^{s-\beta_{K}-1} x_{K}^{\beta_{K}-1} d x_{K} \notag \\
\leq&\  \frac{\Gamma(s)}{\Gamma(s-\beta_{K}) \Gamma(\beta_{K})} \frac{1}{\beta_{K}} \Big(\frac{10K}{\sqrt{N}}\Big)^{\beta_{K}} \Big( 1  + \frac{1}{(1-1/\sqrt{N})^{\abs{s-\beta_{K} -1}}} \Big) \leq C(\bm\beta,K)\delta^{\beta_{K}/2} . \label{eq:proby}
\end{align}
Without loss of generality, we may assume that $h(0) = 0$. Otherwise, we can replace $h(u)$ by   $h(u) - h(0)$ without affecting the value of $\E h(U) - \E A h(Y)$. Combining Lemma~\ref{lem:secondterm} with    \eqref{eq:proby} yields 
\begin{align*}
 \E \big( \E h(\tU) - A h(\tZ)- \delta\cG_{\tZ} A f_h(\tZ)\big) 1(\tZ  \not \in \bs_{N}) \leq&\ C(\bm\beta,K) \big( 1 + \delta^{-2}B_{1}(f_h) \big) \IP(\tZ  \not \in \bs_{N})  \\
 \leq&\ C(\bm\beta,K) \delta^{\beta_{K}/2} \big( 1 + \delta^{-2}B_{1}(f_h) \big)\\
 \leq&\ C(\bm\beta,K) \delta^{\beta_{K}/2}( 1 + \delta^{- 2}),
\end{align*}
where in the last inequality we used the Stein factor bound from Lemma~\ref{lem:steinfactors}.
\end{proof}

\appendix
\section{The interpolator $A$}
\label{app:interpolator}
The operator $A$ discussed in this section is identical to the one introduced in Appendix~A of \cite{Brav2022}. We repeat some its key properties, originally presented in \cite{Brav2022}, as they are needed for the proof of Theorem~\ref{thm:main}. We also present several useful technical results about $A$ that  are not found in \cite{Brav2022}. Namely, Lemma~\ref{lem:lipschitz}, Lemma~\ref{lem:productrule}, Proposition~\ref{prop:interperror}, and Corollary~\ref{cor:polynomials}.

Building on the discussion in Section~\ref{sec:interpolatorinbody}, for a one-dimensional function $f: \delta \Z  \to \R$  we define 
\begin{align*}
A f(x) = \sum_{i=0}^{4} \alpha^{k(x)}_{k(x)+i}(x)f(\delta(k(x)+i)),
\end{align*}
where $k(x) = \lfloor x/\delta\rfloor$ and $\alpha_{k+i}^{k}: \R \to \R$ are weights.   \cite{Brav2022} described how to choose these weights to make $Af(x)$  coincide with $f(\cdot)$ on grid points, and also to make it a differentiable function whose  derivatives  behave like the corresponding finite differences of  $f(\cdot)$.   The idea can be applied to multidimensional grid-valued functions  as well. The following result is Theorem 2 of \cite{Brav2022}. We use this as an interface that contains the important properties of $A$ without delving into the low-level details behind its construction.  
\begin{theorem}
\label{thm:interpolant_def} 
Given a convex set $K \subset \R^{d}$,  define 
\begin{align*}
K_{4} = \{\tx \in K \cap \delta \Z^{d} : \delta(k(\tx)+ \ti) \in K \cap \delta \Z^{d} \text{ for all } 0 \leq \ti \leq 4\te\},
\end{align*} 
where $k(\tx)$ by  $k_{j}(\tx) = \lfloor x_j/\delta\rfloor$. Let $\text{Conv}(K_4)$ be the convex hull of $K_4$, and, for $\tx \in \R^{d}$. There exist weights $\big\{\alpha_{k+i}^{k} : \R \to \R,\  k \in \Z,\ i = 0,1,2,3,4\big\}$  such that for any  $f: K \cap  \delta \Z^{d} \to \R$, the function
\begin{align}
A f(\tx) =&\  \sum_{i_d = 0}^{4} \alpha_{k_d(\tx)+i_d}^{k_d(\tx)}(x_d)\cdots \sum_{i_1 = 0}^{4} \alpha_{k_1(\tx)+i_1}^{k_1(\tx)}(x_1) f(\delta(k(\tx)+\ti)) \notag \\
=&\ \sum_{i_1, \ldots, i_d = 0}^{4} \bigg(\prod_{j=1}^{d} \alpha_{k_j(\tx) +i_j}^{k_j(\tx)   }(x_j)\bigg) f(\delta (k(\tx)+\ti)) , \quad \tx \in \text{Conv}(K_4)  \label{eq:af2}
\end{align}
satisfies $A f(\tx) \in C^{3}(\text{Conv}(K_4))$, where $\ti = (i_1, \ldots, i_d)$ in \eqref{eq:af2}. Additionally, $Af(\tx)$ is infinitely differentiable almost everywhere on $\text{Conv}(K_4)$,
\begin{align}
A f(\delta \tk) = f(\delta \tk), \quad \delta \tk \in K_{4}, \label{eq:interpolates2}
\end{align} 
and  there exists a constant $C(d) > 0$ independent of $f(\cdot)$, $\tx$, and $\delta$, such that  
\begin{align}
\big|  D^\ta Af(\tx) \big|   \leq&\  C(d) \delta^{-\norm{\ta}_{1}}  \max_{0\leq \ti \leq 4\te - \ta} \abs{\Delta^{\ta}f(\delta (k(\tx)+\ti))}, \quad \tx \in \text{Conv}(K_4), \label{eq:multibound2}
\end{align}
for $0 \leq \norm{\ta}_{1} \leq 3$, and \eqref{eq:multibound2} also holds when $\norm{\ta}_{1} = 4$  for almost all $\tx \in  \text{Conv}(K_4)$. Additionally, the weights  $\big\{\alpha_{k+i}^{k} : \R \to \R,\  k \in \Z,\ i = 0,1,2,3,4\big\}$
are degree-$7$ polynomials in $(x-\delta k)/ \delta$ whose coefficients do not depend on $k$ or $\delta$. They satisfy  
\begin{align}
&\alpha_{k}^{k}(\delta k) = 1, \quad \text{ and } \quad \alpha_{k+i}^{k} (\delta k) = 0, \quad &k \in \Z,\ i = 1,2,3,4, \label{eq:alphas_interpolate} \\
&\sum_{i=0}^{4} \alpha^{k}_{k+i}(x) = 1, \quad &k \in \Z,\ x \in \R, \label{eq:weights_sum_one}
\end{align}
and also the following translational invariance property:
\begin{align}
\alpha^{k+j}_{k+j+i}(x+ \delta j)  = \alpha^{k}_{k+i}(x),\quad i,j,k \in \Z, \ x \in \R. \label{eq:weights}
\end{align}
\end{theorem}
From \eqref{eq:af2} we see that $A$ is a linear operator, and  \eqref{eq:weights_sum_one} implies that $A$ applied to a constant simply equals that constant. The following corollary follows from the fact that the weights  $\big\{\alpha_{k+i}^{k} : \R \to \R,\  k \in \Z,\ i = 0,1,2,3,4\big\}$
are degree-$7$ polynomials in $(x-\delta k)/ \delta$ whose coefficients do not depend on $k$ or $\delta$.
\begin{corollary}
\label{cor:boundweights}
There exists $C > 0$ independent of $\delta$ such that for all $x \in \R$ and all $0 \leq i \leq 4$, 
\begin{align*}
\abs{\alpha_{k(x) + i}^{k(x)}(x)} \leq C.
\end{align*}
\end{corollary}
We now present three useful properties of $A$. While Theorem~\ref{thm:interpolant_def} only guarantees that $A f(\tx)$ is thrice continuously differentiable,   we  often need to control the  fourth order remainder term in the Taylor expansion of $Af(\tx)$.  The following lemma, which would have been trivial if $Af(\tx)$ were four-times continuously differentiable, helps with this.  Define 
\begin{align}
\mathcal{D}^{d} = \Big\{f : \R^{d} \to \R : \text{ for all $\tx, \ty \in \R^{d}$, }  \abs{f(\tx) - f(\ty)} \leq \norm{\tx-\ty}_{1} \sup_{\substack{ \min(\tx,\ty) \leq \tz \\ \tz \leq \max(\tx,\ty) \\ \norm{\ta}_{1} = 1}} \abs{D^{\ta} f(\tz)} \Big\}, \label{eq:ddef}
\end{align}
where the minimum and maximum are taken elementwise.  
\begin{lemma}
\label{lem:lipschitz}
 For any $f: \delta \Z^{d} \to \R$, let $Af(\tx)$ be as defined in \eqref{eq:af2}. Then $D^{\ta} Af(\tx) \in \mathcal{D}^{d}$ for any $\ta > 0$ with $\norm{\ta}_{1} = 3$.
\end{lemma} 
The second lemma is a useful identity for applying $A$ to products of functions. 
\begin{lemma}
\label{lem:productrule}
Given $f,g: \Z^{d} \to \R$ let $h(\tu) = f(\tu) g(\tu)$. There exists $\epsilon_{p}: \R^{d} \to \R$ and a constant $C(d)$ such that 
\begin{align*}
A h(\tx) =&\ A f(\tx) A g(\tx) + \epsilon_{p}(\tx)   \quad \text{ and } \\
\abs{\epsilon_{p}(\tx)} \leq&\  C(d) \max_{\substack{\norm{\ta}_{1}=1 \\ 0 \leq i \leq 4\te -\ta}} \abs{\Delta^{\ta} g(\delta (k(\tx) + \ti))}  \max_{\substack{\norm{\ta}_{1}=1 \\ 0 \leq i \leq 4\te -\ta}} \abs{\Delta^{\ta} f(\delta (k(\tx) + \ti))}.
	\end{align*}
\end{lemma} 
For our third result, let $f(\tx)$ be a function defined for all $\tx \in \R^{d}$  and let $f(\tu)$ denote its restriction to $\delta \Z^{d}$. Proposition~\ref{prop:interperror} provides an upper bound on how well $A f(\tx)$ approximates $f(\tx)$. The smoother the function $f(\tx)$, the higher the exponent of $\delta$ in the error bound. 
\begin{proposition}
\label{prop:interperror}
Suppose that $f \in C^{s-1}(\R^{d})$ for some $s \in \{1,2,3,4\}$ and that  $D^{\ta} f(\tx) \in \mathcal{D}^{d}$  when $\norm{\ta}_{1} = s-1$. Then  
\begin{align}
\abs{f(\tx) - A f(\tx)} \leq C(d) \delta^s  \max_{\norm{\ta}_{1} = s} \sup_{0 \leq \tz \leq 4\delta \te } \abs{D^{\ta} f(\delta k(\tx) + \tz)}, \quad \tx \in \R^{d}, \label{eq:afapprox}
\end{align}
where $C(d)$ depends on $d$ but not on $f(\tx)$ or $A  f(\tx)$. 
\end{proposition} 
The following corollary plays an important role in the proof of Lemma~\ref{lem:agxingredients}.
\begin{corollary}
\label{cor:polynomials}
If $f: \R^{d} \to \R$ is a polynomial of degree at most three, then $Af(\tx) = f(\tx)$. 
\end{corollary}
\begin{proof}[Proof of Corollary~\ref{cor:polynomials}]
If $f(x)$ is a polynomials up to the third order, then its fourth-order derivatives are zero. The result follows from applying Proposition~\ref{prop:interperror} with $s = 4$.  
\end{proof}
\begin{proof}[Proof of Lemma~\ref{lem:lipschitz}]
Given $\tx, \ty \in \R^{d}$, define 
\begin{align*}
\tx^{(j)} = (x_1, \ldots, x_{j-1}, y_{j}, \ldots, y_{d}), \quad 1 \leq j \leq d,
\end{align*}
and note that $\tx^{(1)} = \ty$ and $\tx^{(d)} = \tx$. Fix $\ta \in \Z^{d}$ with $\ta > 0$ and $\norm{\ta}_{1} =3$. Then  
\begin{align*}
D^{\ta} A f(\tx) - D^{\ta} A f(\ty) =&\ D^{\ta} A f(\tx^{(d)}) - D^{\ta} A f(\tx^{(1)})= \sum_{j=1}^{d-1} D^{\ta} A f(\tx^{(j+1)}) - D^{\ta} A f(\tx^{(j)}).
\end{align*}
 We now show that 
\begin{align*}
\abs{D^{\ta} A f(\tx^{(j+1)}) - D^{\ta} A f(\tx^{(j)})} \leq \abs{\tx_{j} - \ty_{j}} \sup_{\substack{ \min(\tx,\ty) \leq \tz \\ \tz \leq \max(\tx,\ty)}} \abs{D^{\ta} A f(\tz)}.
\end{align*} 
Suppose that $j = d-1$; the argument is similar for other values of $j$. 
Note that 
\begin{align*}
&D^{\ta} A f(\tx) \\
=&\ \partial_{d}^{a_{d}} \cdots \partial_{1}^{a_1} \bigg(\sum_{i_d = 0}^{4} \alpha_{k_d(\tx)+i_d}^{k_d(\tx)}(x_d)\cdots \sum_{i_1 = 0}^{4} \alpha_{k_1(\tx)+i_1}^{k_1(\tx)}(x_1) f(\delta(k(\tx)+\ti))\bigg) \\
=&\  \partial_{d}^{a_d}\sum_{i_d = 0}^{4} \alpha_{k_d(x_d)+i_d}^{k_d(x_d)}(x_d) \\
 & \hspace{1.5cm} \times \bigg[\partial_{d-1}^{a_{d-1}}\bigg(\sum_{i_{d-1} = 0}^{4} \alpha_{k_{d-1}(x_{d-1})+i_{d-1}}^{k_{d-1}(x_{d-1})}(x_{d-1})\cdots  \partial_{1}^{a_1}\bigg(\sum_{i_1 = 0}^{4} \alpha_{k_1(x_1)+i_1}^{k_1(x_1)}(x_1) f(\delta(k(\tx)+\ti)) \bigg)\bigg)\bigg],
\end{align*}
where the first equality follows by the definition of $A f(\tx)$ in \eqref{eq:af2}.
Treating $x_1, \ldots, x_{d-1}$ as fixed, let us consider the term inside the square brackets as a one-dimensional function in the $d$th dimension. To be precise,   we define $g_{x_1,\ldots,x_{d-1}} : \delta \Z \to \R$ as 
\begin{align*}
g_{x_1,\ldots,x_{d-1}}(\delta \ell)  =&\ \partial_{d-1}^{a_{d-1}}\bigg(\sum_{i_{d-1} = 0}^{4} \alpha_{k_{d-1}(x_{d-1})+i_{d-1}}^{k_{d-1}(x_{d-1})}(x_{d-1})\cdots  \partial_{1}^{a_1}\bigg(\sum_{i_1 = 0}^{4} \alpha_{k_1(x_1)+i_1}^{k_1(x_1)}(x_1)\\
& \hspace{3cm} \times  f\big(\delta(k(x_1,\ldots,x_{d-1},0)+(i_1,\ldots,i_{d-1},0)) + \delta \ell \big) \bigg)\bigg), \quad \ell \in \Z.
\end{align*}
Then
\begin{align*}
D^{\ta} A f(\tx)  
=  \partial_{d}^{a_d}\sum_{i_d = 0}^{4} \alpha_{k_d(x_d)+i_d}^{k_d(x_d)}(x_d) g_{x_1,\ldots,x_{d-1}}\big( \delta (k_{d}(x_d) + i_{d}) \big) = \partial^{a_d} A g_{x_1,\ldots,x_{d-1}}(x_{d})
\end{align*}
 Theorem~\ref{thm:interpolant_def} says that $A g_{x_1,\ldots,x_{d-1}}(x_d)$ is infinitely differentiable almost everywhere. Thus,  
\begin{align*}
&\abs{D^{\ta} A f(\tx^{(d)}) - D^{\ta} A f(\tx^{(d-1)})} \\
=&\ \abs{\partial^{a_d}  A g_{x_1,\ldots,x_{d-1}}(x_{d}) - \partial^{a_d}  A g_{x_1,\ldots,x_{d-1}}(y_{d})} \\
=&\  \abs{\int_{y_d}^{x_d} \partial^{a_d+1} A g_{x_1,\ldots,x_{d-1}}(x')dx'} \\
\leq&\ \abs{x_d - y_d} \sup_{(x_d \wedge y_d) \leq z_d \leq (x_d \vee y_{d}) } \abs{\partial^{a_d+1} A g_{x_1,\ldots,x_{d-1}}(z_d)} \\
\leq&\ \abs{x_d - y_d} \max_{\norm{\ta'}_{1} = \norm{\ta}_{1} + 1} \sup_{\substack{ \min(\tx,\ty) \leq \tz \\ \tz \leq \max(\tx,\ty)}} \abs{D^{\ta'} A f(z)},
\end{align*}
where in the last inequality we used $\partial^{a_d+1} A g_{x_1,\ldots,x_{d-1}}(x_{d}) = \partial_{d} D^{\ta} A f(\tx)$. 
\end{proof}

%
\begin{proof}[Proof of Lemma~\ref{lem:productrule}]
Using \eqref{eq:af2} of Theorem~\ref{thm:interpolant_def}, for $\tx \in \R^{d}$, 
\begin{align*}
A h(\tx) =&\ \sum_{i_1, \ldots, i_d = 0}^{4}  \prod_{j=1}^{d} \alpha_{k_j(\tx) +i_j}^{k_j(\tx)   }(x_j)   f(\delta(k(\tx)+\ti)) g(\delta(k(\tx)+\ti)) \\
=&\  \sum_{i_1, \ldots, i_d = 0}^{4}  \prod_{j=1}^{d} \alpha_{k_j(\tx) +i_j}^{k_j(\tx)   }(x_j)  A f(\tx)  g(\delta(k(\tx)+\ti))\\
&+ \sum_{i_1, \ldots, i_d = 0}^{4}  \prod_{j=1}^{d} \alpha_{k_j(\tx) +i_j}^{k_j(\tx)   }(x_j)  \big(f(\delta(k(\tx)+\ti)) - A f(\tx)\big)   g(\delta(k(\tx)+\ti)).
\end{align*}
By the definition of $A$, 
\begin{align*}
&\sum_{i_1, \ldots, i_d = 0}^{4}  \prod_{j=1}^{d} \alpha_{k_j(\tx) +i_j}^{k_j(\tx)   }(x_j)  A f(\tx)  g(\delta(k(\tx)+\ti)) = Af(\tx) Ag(\tx).
\end{align*}
Furthermore, 
\begin{align*}
& \sum_{i_1, \ldots, i_d = 0}^{4}  \prod_{j=1}^{d} \alpha_{k_j(\tx) +i_j}^{k_j(\tx)   }(x_j)  \big(f(\delta(k(\tx)+\ti)) - A f(\tx)\big)   g(\delta(k(\tx)+\ti)) \\ 
=&\ \sum_{i_1, \ldots, i_d = 0}^{4}  \prod_{j=1}^{d} \alpha_{k_j(\tx) +i_j}^{k_j(\tx)   }(x_j)  \big(f(\delta(k(\tx)+\ti)) - A f(\tx)\big)   \big(g(\delta(k(\tx)+\ti)) - g(\delta k(\tx))\big),
\end{align*}
because again, by the definition of $A$,  
\begin{align*}
 \sum_{i_1, \ldots, i_d = 0}^{4}  \prod_{j=1}^{d} \alpha_{k_j(\tx) +i_j}^{k_j(\tx)   }(x_j)  \big(f(\delta(k(\tx)+\ti)) - A f(\tx)\big)  g(\delta k(\tx)) = g(\delta k(\tx)) (Af(\tx) - A f(\tx)) = 0.
\end{align*}
Setting 
\begin{align*}
\epsilon_{p}(\tx) = \sum_{i_1, \ldots, i_d = 0}^{4}  \prod_{j=1}^{d} \alpha_{k_j(\tx) +i_j}^{k_j(\tx)   }(x_j)  \big(f(\delta(k(\tx)+\ti)) - A f(\tx)\big)   \big(g(\delta(k(\tx)+\ti)) - g(\delta k(\tx))\big), 
\end{align*}
we have shown that $A h(\tx) = Af(\tx) A g(\tx) + \epsilon_{p}(\tx)$. To bound $\epsilon_{p}(\tx)$, observe that Corollary~\ref{cor:boundweights} implies that 
\begin{align*}
\abs{\alpha_{k_j(\tx) + i_j}^{k_j(\tx)}(x_j)} \leq C.
\end{align*}
Furthermore, 
\begin{align*}
\abs{g(\delta(k(\tx)+\ti)) - g(\delta k(\tx))} \leq&\ C(d) \max_{\substack{\norm{\ta}_{1}=1 \\ 0 \leq \tj \leq 4\te-\ta }}\abs{\Delta^{\ta} g(\delta (k(\tx) + \tj) }
\end{align*}
since $\ti \leq 4e$, and  
\begin{align*}
\abs{f(\delta(k(\tx)+\ti)) - A f(\tx)} \leq&\ \abs{f(\delta (k(\tx)+\ti)) - f(\delta k(\tx))} + \abs{f(\delta k(\tx)) - A f(\tx)} \\
\leq&\ C(d) \max_{\substack{\norm{\ta}_{1}=1 \\ 0 \leq j \leq 4\te -\ta }}\abs{\Delta^{\ta} f(\delta (k(\tx) + \tj) } + \abs{f(\delta k(\tx)) - A f(\tx)},
\end{align*}
and 
\begin{align*}
\abs{f(\delta k(\tx)) - A f(\tx)}  = \abs{A f(\delta k(\tx) ) - A f(\tx)} \leq&\ C(d) \delta \sup_{\substack{\tz \in [\delta k(\tx),\tx]\\ \norm{\ta}_{1}=1}} \abs{D^{\ta} A f(\tz)}  \\
\leq&\ C(d) \max_{\substack{\norm{\ta}_{1}=1 \\ 0 \leq i \leq 4\te -\ta}} \abs{\Delta^{\ta} f(\delta (k(\tx) + \ti))},
\end{align*}
where the last inequality is due to \eqref{eq:multibound2} of Theorem~\ref{thm:interpolant_def}.
\end{proof}

Fix a function $f(\tx)$ satisfying the conditions of Proposition~\ref{prop:interperror}. Before we prove Proposition~\ref{prop:interperror} we require the following key Lemma. 
\begin{lemma}
\label{lem:Dapprox}
For any $s \in \{1,2,3,4\}$ and $1 \leq \norm{\ta}_{1} \leq s-1$, there exists a function $E: \R^{d} \to \R$ such that
\begin{align*}
D^{\ta} A f(\delta k(\tx)) =&\ D^{\ta} f(\delta k(\tx)) + \delta^{s-\norm{\ta}_{1}} E(\delta k(\tx)), \quad \tx \in \R^{d}
\end{align*}
and 
\begin{align*}
\abs{E(\tx)} \leq&\  C(d)\max_{\norm{\ta}_{1} = s} \sup_{0\leq \tz \leq    4\delta \te } \abs{D^{\ta} f(\tx + \tz)}.
\end{align*}
\end{lemma}
Throughout this section, we use $E(\tx)$ to denote any function from $\R^{d} \to \R$ that  satisfies 
\begin{align*}
\abs{E(\tx)} \leq  C(d)\max_{\norm{\ta}_{1} = s} \sup_{0\leq \tz \leq    4\delta \te } \abs{D^{\ta} f(\tx + \tz)},
\end{align*}
where $C(d) > 0$ depends only on $d$ and not $f(\tx)$.  
\begin{proof}[Proof of Proposition~\ref{prop:interperror}]
  Expanding both $f(\tx)$ and $A f(\tx)$ around $\delta k(\tx)$ yields
\begin{align*}
  A f(\tx) - f(\tx)  =&\  \sum_{j = 1}^{s-1} \frac{1}{j!} \sum_{\ta : \norm{\ta}_{1} = j} \bigg(\prod_{i=1}^{d} (x_{i} - \delta k_{i}(\tx))^{a_{i}} \bigg)\big(D^{\ta} A f(\delta k(\tx)) - D^{\ta} f(\delta k(\tx))\big)  \\  
    &+ \frac{1}{(s-1)!} \sum_{\ta : \norm{\ta}_{1} = s-1} \bigg(\prod_{i=1}^{d} (x_{i} - \delta k_{i}(\tx))^{a_{i}} \bigg) \big( D^{\ta} A f(\bm{\xi}_{1}) -  D^{\ta} Af(\delta k(\tx))\big) \\
 &+ \frac{1}{(s-1)!} \sum_{\ta : \norm{\ta}_{1} = s-1}   \bigg(\prod_{i=1}^{d} (x_{i} - \delta k_{i}(\tx))^{a_{i}} \bigg) \big( D^{\ta}   f(\bm{\xi}_{2}) -  D^{\ta} f(\delta k(\tx))\big),
\end{align*}
where $\bm{\xi}_1,\bm{\xi}_2 \in [\delta k(\tx),\tx] \subset [\delta k(\tx) , \delta (k(\tx) + 1))$. Since $\abs{x_i - \delta k_i(\tx)} \leq \delta$, Lemma~\ref{lem:Dapprox} implies that the first term on the right-hand side equals $\delta^s E(\delta k(\tx))$. The second term equals $\delta^{s} E(\delta k(\tx))$  because $\delta k(\tx) \leq \xi_{1} \leq \tx < \delta(k(\tx) + 1)$, and because $D^{\ta} Af(x) \in \mathcal{D}^{\ta}$ when $\norm{\ta}_{1} = s-1$; the latter fact follows from Lemma~\ref{lem:lipschitz} if $s = 4$, and from the fact that $A f(\tx) \in C^{3}(\R^{d})$ if $s < 4$. The last line equals $\delta^4 E(\delta k(\tx))$ by our assumption that $D^{\ta} f(\tx) \in \mathcal{D}^{d}$  when $\norm{\ta}_{1} = s-1$. 
\end{proof}
To prove Lemma~\ref{lem:Dapprox}, for $f: \R^{d} \to \R$,  $1 \leq i \leq d$, and $\tx \in \R^{d}$, we define $\Delta_{i}^{0} f(\tx) = f(\tx)$, $\widetilde \Delta_{i}^{(0)} f(\tx) = f(\tx)$,  
\begin{align*}
\Delta_{i} f(\tx) =&\ f(\tx+\delta \te_{i} - f(\tx),\\
\widetilde \Delta_{i}^{(1)} f(\tx)=&\ \Big(\Delta_{i} - \frac{1}{2} \Delta_{i}^2 + \frac{1}{3} \Delta_{i}^3 \Big)f(\tx),\\
\widetilde \Delta_{i}^{(2)} f(\tx)=&\ \big(\Delta_{i}^2 - \Delta_{i}^3 \big) f(\tx), \quad \widetilde \Delta_{i}^{(3)} f(\tx)=   \Delta_{i}^3  f(\tx).
\end{align*}
We prove the following result after proving Lemma~\ref{lem:Dapprox}. 
\begin{lemma}
\label{lem:dform}
Given $f: \R^{d} \to \R$ and the corresponding $A f(\tx)$, for $1 \leq \norm{\ta}_{1} \leq 3$, 
\begin{align*}
 D^{\ta} A f(\delta k(\tx)) =&\ \delta^{-\norm{\ta}_{1}} \widetilde \Delta^{(a_1)}_{1} \cdots \widetilde \Delta_{d}^{(a_d)} f(\delta k(\tx)), \quad \tx \in \R^{d}.
\end{align*}
\end{lemma}
\begin{proof}[Proof of Lemma~\ref{lem:Dapprox}]
Suppose we have shown that for $1 \leq i \leq d$ and $1 \leq j \leq s-1$, 
\begin{align}
\widetilde \Delta_{i}^{(j)} f(\tx) =&\ \delta^{j}\partial_{i}^{j} f(\tx) + \delta^{s}  E_{i}(\tx), \label{eq:midtilde}
\end{align}
where $ E_{i}(\tx)$ is a generic function satisfying
\begin{align*}
\abs{ E_{i}(\tx)} \leq&\  C \sup_{\tx \leq \tz \leq  \tx + 4\delta  \te_{i}} \abs{\partial_{i}^{s} f(\tz)}, \quad \tx \in \R^{d}
\end{align*}
for some constant $C > 0$. Combining \eqref{eq:midtilde} with Lemma~\ref{lem:dform} yields  
\begin{align*}
 D^{\ta} A f(\delta k(\tx)) =&\ \delta^{-\norm{\ta}_{1}} \widetilde \Delta^{(a_1)}_{1} \cdots \widetilde \Delta_{d}^{(a_d)} f(\delta k(\tx)) \notag\\
=&\ \delta^{-\norm{\ta}_{1}} \widetilde \Delta^{(a_1)}_{1} \cdots \widetilde \Delta_{d-1}^{(a_{d-1})} \Big( \delta^{a_d} \partial_{d}^{a_d} f(\delta k(\tx)) + \delta^{s}  E_{d}(\delta k(\tx))  \Big) \notag \\
=&\ \delta^{-\norm{\ta}_{1}} \delta^{a_1} \cdots \delta^{a_d} D^{\ta} f(\delta k(\tx)) +   \delta^{s} \delta^{-\norm{\ta}_{1}} \widetilde \Delta^{(a_1)}_{1} \cdots \widetilde \Delta_{d-1}^{(a_{d-1})}  E_{d}(\delta k(\tx)) \\
=&\  D^{\ta} f(\delta k(\tx)) + \delta^{s-\norm{\ta}_{1}} E(\delta k(\tx)). 
\end{align*}
To justify the last equality, note that for $j \neq d$, 
\begin{align*}
 \abs{\Delta_{j} E_{d}(\delta k(\tx))} =\abs{E_{d}(\delta k(\tx) + \delta \te_{j}) - E_{d}(\delta k(\tx))} 
\leq  \abs{E_{d}(\delta k(\tx) + \delta \te_{j})} + \abs{E_{d}(\delta k(\tx))} 
= \delta E(\delta k(\tx)).
\end{align*} 
Similarly, applying $\Delta^{(a_1)}_{1} \cdots \widetilde \Delta_{d-1}^{(a_{d-1})}$ to $E_{d}(\delta k(\tx))$ also results in $E(\delta k(\tx))$. We now prove \eqref{eq:midtilde}. Suppose that $s = 4$.  The reader can verify, using Taylor expansion, that 
\begin{align}
\Delta_{i} f(\tx) =&\  \partial_{i} f(\tx) + \frac{1}{2}\partial_{i}^{2} f(\tx) + \frac{1}{6}\partial_{i}^{3} f(\tx) + \delta^{4} E_{i}(\tx), \notag \\
\Delta_{i}^2 f(\tx) =&\  \delta^2 \partial_{i}^2 f(\tx) + \delta^3 \partial_{i}^3 f(\tx) + \delta^4 E_{i}(\tx), \notag \\
\Delta_{i}^3 f(\tx) =&\ \delta^3 \partial_{i}^3 f(\tx) + \delta^4 E_{i}(\tx), \label{eq:basictaylor}
\end{align}
which, when combined with the definition of $\widetilde \Delta_{i}^{(j)} f(\tx)$, immediately implies \eqref{eq:midtilde}. When $s < 4$, the proof of \eqref{eq:midtilde} is similar, except that we need to use a lower order Taylor expansion in \eqref{eq:basictaylor}. For example, when $s = 3$, we would use 
\begin{align*}
\Delta_{i} f(\tx) =&\  \partial_{i} f(\tx) + \frac{1}{2}\partial_{i}^{2} f(\tx) +   \delta^{3} E_{i}(\tx) \quad \text{ and } \quad \Delta_{i}^2 f(\tx) =   \delta^2 \partial_{i}^2 f(\tx) +  \delta^3 E_{i}(\tx).
\end{align*}
\end{proof}
We conclude this section by proving Lemma~\ref{lem:dform}.
\begin{proof}[Proof of Lemma~\ref{lem:dform}]
Let us first consider the case when $d  = 1$. 
According to the original definition of $A f(x)$ in Theorem~1 of \cite{Brav2022},
\begin{align*}
A f(x) = P_{k(x)}(x), \quad x \in \R,
\end{align*}
where 
\begin{align}
P_{k}(x) =&\ f(\delta k) + \Big(\frac{x-\delta k}{\delta} \Big)(\Delta - \frac{1}{2}\Delta^2 + \frac{1}{3}\Delta^3) f(\delta k)  \notag  \\
&+ \frac{1}{2} \Big(\frac{x-\delta k}{\delta} \Big)^2\big(\Delta^2 - \Delta^3\big) f(\delta k) + \frac{1}{6}  \Big(\frac{x-\delta k}{\delta} \Big)^3 \Delta^3 f(\delta k)\notag \\
& -\frac{23}{3} \Big(\frac{x-\delta k}{\delta} \Big)^4 \Delta^4 f(\delta k)  +\frac{41}{2} \Big(\frac{x-\delta k}{\delta} \Big)^5\Delta^4 f(\delta k)  \notag  \\
&- \frac{55}{3}  \Big(\frac{x-\delta k}{\delta} \Big)^6\Delta^4 f(\delta k)  +\frac{11}{2} \Big(\frac{x-\delta k}{\delta} \Big)^7\Delta^4 f(\delta k) , \quad   x \in \R. \label{eq:pplus}
\end{align}
We can also write $A f(x)$ as the weighted sum
\begin{align*}
A f(x) =  \sum_{i=0}^{4} \alpha^{k(x)}_{k(x)+i}(x)f(\delta(k(x)+i)),
\end{align*}
where the weights $\alpha^{k}_{k+i}(x)$ are defined by the corresponding polynomial $P_{k}(x)$. Now $A f \in C^{3}(\R)$ by  Theorem~\ref{thm:interpolant_def}, meaning that $\partial^{j} A f(\delta k(x))$
  equals the corresponding right-derivative of $P_{k(x)}(x)$ at $x = \delta k(x)$ for $1 \leq j \leq 3$. The result follows immediately by inspecting the derivatives of \eqref{eq:pplus}. 
  
When $d > 1$, the proof is similar.  Let us write $\tk$ instead of $k(\tx)$ for notational convenience. Since $A f(\tx) \in C^{3}(\R^{d})$ by Theorem~\ref{thm:interpolant_def}, we fix $\ta$ with $1 \leq \norm{\ta}_{1} \leq 3$ and consider 
\begin{align*}
D^{\ta} A f(\tx) =&\ D^{\ta}\Bigg( \sum_{i_d = 0}^{4} \alpha_{k_d+i_d}^{k_d}(x_d)\cdots \sum_{i_1 = 0}^{4} \alpha_{k_1+i_1}^{k_1}(x_1) f(\delta(\tk+\ti))\Bigg) \\
=&\  \Bigg(\sum_{i_d = 0}^{4} \partial_{d}^{a_d}\alpha_{k_d+i_d}^{k_d}(x_d)\cdots \bigg(\sum_{i_2 = 0}^{4} \partial_{2}^{a_2}\alpha_{k_2+i_2}^{k_2}(x_2)\bigg(\sum_{i_1 = 0}^{4} \partial_{1}^{a_1}\alpha_{k_1+i_1}^{k_1}(x_1) f(\delta(\tk+\ti)) \bigg)\bigg)\Bigg).
\end{align*}
The first equality follows from \eqref{eq:af2} of Theorem~\ref{thm:interpolant_def}. If we think of the innermost sum as a  one-dimensional function in $x_1$ only, it follows that
\begin{align*}
\sum_{i_1 = 0}^{4} \partial_{1}^{a_1}\alpha_{k_1+i_1}^{k_1}(\delta k_1) f(\delta(\tk+\ti)) = \delta^{-a_1}\widetilde \Delta^{(a_1)}_{1} f\big(\delta (\tk+(0,i_2,\ldots,i_{d}))\big).
\end{align*}
Proceeding identically along each of the remaining dimensions yields the result. 
\end{proof}

Before concluding this section, we prove Lemmas~\ref{lem:convdet} and~\ref{lem:secondterm}.
\begin{proof}[Proof of Lemma~\ref{lem:convdet}]
Fix $h \in \mathcal{M}_{4}$ and note that 
\begin{align*}
\abs{\E h(\tV) - \E h(\tV')} \leq&\ \abs{\E h(\tV) - \E A h(\tV')} + \abs{\E A h(\tV') - \E h(\tV')},
\end{align*}
where $A h(\cdot)$ is understood to be the operator $A$ applied to the restriction of $h(\tx)$ to $\delta \Z^{d}$. Proposition~\ref{prop:interperror} and the fact that $h \in \mathcal{M}_{4}$ imply that there exists some $C'(d) > 0$ such that
\begin{align*}
\abs{\E A h(\tV') - \E h(\tV')} \leq C'(d) \delta^{4} \max_{\norm{\ta}_{1}=4} \norm{D^{\ta} h} \leq C'(d) \delta^4.
\end{align*}
Furthermore, as argued at the end of the proof of Lemma~1 of \cite{Brav2022}, there exists some constant $C$ such that the restriction of $h(\tx)$ to $\delta \Z^{d}$ belongs to $\mathcal{M}_{disc,4}(C)$, implying that 
\begin{align*}
\abs{\E h(\tV) - \E A h(\tV')} \leq \sup_{\substack{h \in \mathcal{M}_{disc,4}(C)  }} \abs{\E h(\tV) - \E Ah(\tV')}.
\end{align*}
\end{proof}

\begin{proof}[Proof of Lemma~\ref{lem:secondterm}]
 Since $h \in \mathcal{M}_{disc,4}(C)$ implies that $\abs{ h(\tu + \delta \te_{i}) -  h(\tu)} \leq C \delta$ for all $1 \leq i \leq K-1$, it follows that
\begin{align*}
\abs{h(\tu)} \leq \abs{h(0)} + C \delta  \norm{\tu/\delta}_{1}  \leq \abs{h(0)} + C \norm{\tu}_{1}, \quad \tu \in \delta \Z^{K-1}.
\end{align*}
Furthermore,  \eqref{eq:multibound2} of Theorem~\ref{thm:interpolant_def} implies that the first-order partial derivatives of  $A h(\tx)$ are bounded by $C(K)$, implying that  
\begin{align*}
\abs{A h(\tZ)} \leq \abs{A h(0)} + C(K) \norm{\tZ}_{1} \leq  \abs{h(0)} + C(K),
\end{align*} 
where the last inequality follows from $A h(0) = h(0)$ and $\tZ \in \bs$. To bound $\abs{\cG_{\tZ} A f_h(\tZ)}$, we recall that 
\begin{align*}
\cG_{\tZ} A f_h(\tx) =  \frac{1}{2} \sum_{i=1}^{K-1} (\beta_{i} - s x_{i}) \partial_{i} A f_h(\tx) + \frac{1}{2} \sum_{i,j = 1}^{K-1} x_{i}(\delta_{ij} - x_{j}) \partial_{i}\partial_{j} A f_h(\tx), \quad \tx \in \bs.
\end{align*}
Note that  $\abs{\beta_{i} - s x_{i}} \leq s$, $\abs{ x_{i}(\delta_{ij}- x_{j})} \leq 1$, and that \eqref{eq:multibound2} of Theorem~\ref{thm:interpolant_def} implies that,
\begin{align*} 
\abs{\partial_{i} A f_h(\tx)} \leq&\ C(K) \delta^{-1} \max_{\norm{\ta}_{1} =1} \norm{\Delta^{\ta} f_h} \leq C(K) \delta^{-1}   \norm{  f_h}, \\
\abs{\partial_{i}\partial_{j} A f_h(\tx)} \leq&\ C(K) \delta^{-2} \max_{\norm{\ta}_{1} =2} \norm{\Delta^{\ta} f_h} \leq C(K) \delta^{-2}   \norm{  f_h}.
\end{align*}
To conclude, we recall from \eqref{eq:b1bound} that $\norm{f_h} \leq \delta^{-1} K B_{1}(f_h)$.  
\end{proof}

%

\section{Generator expansion \& Stein factors}
\label{app:genexpansion}
This section is dedicated to proving Lemmas~\ref{lem:agx} and~\ref{lem:steinfactors}.  Recall   that for $\tx \in \R^{K-1}$, we  define $k(\tx)$ by  $k_{j}(\tx) = \lfloor x_j/\delta\rfloor$, that $\delta = 1/N$,    that  $\s_{N}$ is defined in \eqref{eq:ds4}, and that $\bs_{N} =  \text{Conv}(\s_{N})$. We may assume without loss of generality that $100K^2 < N$, which implies that   $\s_{N} \neq \emptyset$, because the claim in Lemma~\ref{lem:agx} holds trivially for all $0 < N \leq 100K^{2}$ (since this covers a finite number of $N$).   We claim that
\begin{align}
\text{ if $\tx \in \bs_{N}$, then $\delta (k(\tx) + \ti) \in \s$ for all $0 \leq \ti \leq 10\te$}. \label{eq:inside}
\end{align}
We argue this as follows. Since $0 \leq \delta k(\tx) \leq \tx$, then  $\delta k(\tx) \in \s_{N}$ for any $\tx \in \bs_{N}$,  because $ \delta \sum_{j=1}^{K-1} k_j(\tx)   \leq \sum_{j=1}^{K-1} x_j$. Thus, for all $0 \leq i \leq 10\te$, 
\begin{align*}
\delta \sum_{j=1}^{K-1} (k_j(\tx) + i_j) \leq 1 - 10K/\sqrt{N} + 10K/N < 1,
\end{align*}
implying \eqref{eq:inside}. Combining \eqref{eq:inside} with \eqref{eq:af2} of Theorem~\ref{thm:interpolant_def} implies that  $A (\cG_{\tU} A f_h)(\tx)$ is well defined if $\tx \in \bs_{N}$.  To derive an expression for $A(\cG_{\tU} A f_h)(\tx)$, we define 
\begin{align*}
&b_{i}(\tu) = \E_{u}(U_i(1) - u_i), \\
&a_{ij}(\tu) =    \E_{u}(U_i(1) - u_i)(U_j(1) - u_j), \\
&c_{ijk}(\tu) =  \E_{u}(U_i(1) - u_i)(U_j(1) - u_j)(U_{k}(1)-u_{k}), \\
&\bar d_{ijk\ell}(\tu) = \E_{u} \abs{(U_i(1) - u_i)(U_j(1) - u_j)(U_{k}(1)-u_{k})(U_{\ell}(1)-u_{\ell})} \quad \tu \in \s.
\end{align*}
Since $A f_h  \in C^{3}(\R^{K-1})$, by Theorem~\ref{thm:interpolant_def}, we know that for any $\tu \in \s$, 
\begin{align*}
  \cG_{\tU} (A f_h)(\tu)  =&\ \E_{\tu} A f_h(\tU(1)) - A f_h(\tu) \\
=&\ \sum_{i=1}^{K-1} b_{i}(\tu) \partial_{i} A f_h(\tu) + \frac{1}{2}\sum_{i,j=1}^{K-1} a_{ij}(\tu) \partial_{i}\partial_{j} A f_h(\tu) \\
& \quad  + \frac{1}{6} \sum_{i,j,k=1}^{K-1} c_{ijk}(\tu) \partial_{i}\partial_{j}\partial_{k} A f_h(\tu)  + \epsilon(\tu),
\end{align*}
where 
\begin{align*}
\epsilon(\tu) =&\ \frac{1}{6} \sum_{i,j,k=1}^{K-1}\E_{\tu} \Big( (U_i(1) - u_i)(U_j(1) - u_j)(U_{k}(1)-u_{k}) \big(\partial_{i}\partial_{j}\partial_{k}A f_h(\bm{\xi}) - \partial_{i}\partial_{j}\partial_{k}A f_h(\tu)\big)\Big)
\end{align*}
and $\bm{\xi}$ is between $\tu$ and $\tU(1)$. Since $A$ is a linear operator, it follows that for any $\tx \in \bs_{N}$, 
\begin{align}
A (\cG_{\tU}  A f_h)(\tx) =&\ \sum_{i=1}^{K-1} A(  b_{i}\partial_{i} A f_h)(\tx) + \frac{1}{2}\sum_{i,j =1}^{K-1} A ( a_{ij}\partial_{i}\partial_{j} A f_h)(\tx) \notag \\
&\quad +  \frac{1}{6} \sum_{i,j,k =1}^{K-1} A ( c_{ijk}\partial_{i}\partial_{j}\partial_{k} A f_h)(\tx) + A \epsilon(\tx). \label{eq:AGX}
\end{align}
Let us present several lemmas that we need to analyze the right-hand side.  
\begin{lemma}
\label{lem:moments}
Fix $\tu \in \s$ and let $\bar u_{i} = u_{i}\Sigma  - p_{i}$. Then for $1 \leq i,j \leq K-1$, 
\begin{align*}
b_{i}(\tu) =&\ - \bar u_{i} \quad \text{ and } \quad  a_{ij}(\tu) =  \bar u_{i}\bar u_{j}  + \frac{1}{N}\big(u_{i} - \bar u_{i} \big) \big( \delta_{ij} - ( u_{j}  - \bar u_{j})\big).
\end{align*}  
Furthermore, there exists a constant $C > 0$ such that for any   $1 \leq i,j,k,\ell \leq K-1$, 
\begin{align*}
&\abs{c_{ijk}(\tu)} \leq  C\left( \frac{1}{N} + \Sigma\right)^2, \quad \text{ and } \quad \bar d_{ijk\ell}(\tu) \leq \left( \frac{2}{\sqrt N} + \Sigma\right)^4, \quad \tu \in \s.  
\end{align*}
\end{lemma} 
\noindent The proof of Lemma~\ref{lem:moments} is left to later in this section as it essentially many elementary but tedious moment calculations. Since $b_{i}(\tu)$ and $a_{ij}(\tu)$ are polynomials in $\tu$, we let $b_{i}(\tx)$ and $a_{ij}(\tx)$ be their natural extensions to $\bs$. Furthermore, Corollary~\ref{cor:polynomials} implies that $A b_{i}(\tx) = b_{i}(\tx)$ and $A a_{ij}(\tx) = a_{ij}(\tx)$. We recall the definition of $B_i(\cdot)$ from \eqref{eq:Bi}.
\begin{lemma}
\label{lem:agxingredients}
For any $x \in \bs_{N}$, 
\begin{align*}
A(  b_{i}\partial_{i} A f_h)(\tx) =&\   b_i(\tx) \partial_{i} A f_{h}(\tx) + \tilde \epsilon_{i}(\tx), \\
A(  a_{ij}\partial_{i} A f_h)(\tx) =&\   a_{ij}(\tx)\partial_{i}\partial_{j} A f_{h}(\tx) + \tilde \epsilon_{ij}(\tx), \\
A ( c_{ijk}\partial_{i}\partial_{j}\partial_{k} A f_h)(\tx) =&\ \tilde \epsilon_{ijk}(\tx),
\end{align*}
where 
\begin{align}
\abs{\tilde \epsilon_{i}(\tx)} \leq&\ C(K) \Sigma \big(B_{2}(f_h) + \delta^{-1} B_{4}(f_h) \big), \label{eq:v1}  \\
\abs{\tilde \epsilon_{ij}(\tx)} \leq&\ C(K) (\Sigma^{2} + \delta) (\delta^{-1}  B_{3}(f_h) + \delta^{-2} B_{4}(f_h)),  \label{eq:v2} \\
\abs{\tilde \epsilon_{ijk}(\tx)} \leq&\ C(K) \Big(\delta+ \Sigma\Big)^2 \delta^{-3} B_{3}(f_h) \label{eq:v3}.
\end{align}
Furthermore, 
\begin{align}
&\abs{A \epsilon(\tx)} \leq C(K) \Big(2\sqrt{\delta} + \Sigma\Big)^{4} \delta^{-4} B_{4}(f_h) + C(K) B_{1} (f_h)\frac{  N^{4K+1}}{( \Sigma - p_{K})^{4K}} \Big( 1- \frac{10K(\sqrt{N}-1)}{N} (1-\Sigma)\Big)^{N}. \label{eq:v4}
\end{align}
\end{lemma}
\begin{proof}[Proof of Lemma~\ref{lem:agx}]
Recall  our assumption \eqref{ass:main}, which implies that $p_{i} = \beta_{i}/2N$ and, therefore, $\Sigma = s/2N$. Thus, the form of $\cG_{\tZ} f(\tx)$ in \eqref{eq:gz} yields 
\begin{align*}
\delta\cG_{\tZ} f(\tx) =&\  \frac{1}{2N} \sum_{i=1}^{K-1} (\beta_{i} - s x_{i}) \partial_{i} f(\tx) + \frac{1}{2N} \sum_{i,j = 1}^{K-1} x_{i}(\delta_{ij} - x_{j}) \partial_{i}\partial_{j} f(\tx) \\
=&\  \sum_{i=1}^{K-1} (p_{i} - \Sigma x_{i}) \partial_{i} f(\tx) + \frac{1}{2N} \sum_{i,j = 1}^{K-1} x_{i}(\delta_{ij} - x_{j}) \partial_{i}\partial_{j} f(\tx), \quad \tx \in \bs.
\end{align*}
Combining Lemma~\ref{lem:agxingredients} with \eqref{eq:AGX} yields 
\begin{align*}
A (\cG_{U}  A f_h)(\tx) =&\ \sum_{i=1}^{K-1}  b_{i}(\tx)\partial_{i} A f_h (\tx) + \frac{1}{2}\sum_{i,j =1}^{K-1}  a_{ij}(\tx)\partial_{i}\partial_{j} A f_h (\tx) \\
&+ \sum_{i=1}^{K-1}  \tilde \epsilon_{i}(\tx) + \frac{1}{2}\sum_{i,j =1}^{K-1} \tilde \epsilon_{ij}(\tx) + \frac{1}{6} \sum_{i,j,k =1}^{K-1} \tilde \epsilon_{ijk}(\tx) + A \epsilon(\tx).
\end{align*}
Now, Lemma~\ref{lem:moments} says that   $b_{i}(\tx) =  -x_{i}\Sigma + p_{i}$ and, letting  $\bar x_{i} = x_{i} \Sigma - p_{i}$, that
\begin{align*}
a_{ij}(\tx) =&\    \bar x_{i}\bar x_{j}  + \frac{1}{N}\big(x_{i} - \bar x_{i} \big) \big( \delta_{ij} - ( x_{j}  - \bar x_{j})\big) \\
=&\ \frac{1}{N} x_{i} (\delta_{ij} - x_{j}) - \frac{1}{N}\Big( \bar x_{i}(\delta_{ij} - x_{j}) - x_{i} \bar x_{j} - \bar x_{i} \bar x_{j} \Big) + \bar x_{i}\bar x_{j}.
\end{align*}
Thus, recalling that $\delta = 1/N$, and using $\abs{\bar x_{i}} \leq \Sigma$, it follows that 
\begin{align}
& \bigg|\sum_{i=1}^{K-1}  b_{i}(\tx)\partial_{i} A f_h (\tx) + \frac{1}{2}\sum_{i,j =1}^{K-1}  a_{ij}(x)\partial_{i}\partial_{j} A f_h (\tx) - \delta \cG_{\tZ} A f_h(\tx) \bigg| \notag \\
=&\ \bigg|\frac{1}{2}\sum_{i,j =1}^{K-1}  \Big(- \frac{1}{N}\Big( \bar x_{i}(\delta_{ij} - x_{j}) - x_{i} \bar x_{j} - \bar x_{i} \bar x_{j} \Big) + \bar x_{i}\bar x_{j}  \Big) \partial_{i}\partial_{j} A f_h (\tx)\bigg|  \notag \\
\leq&\ C(K) \big(\delta \Sigma + \Sigma^2\big) \delta^{-2} B_{2} (f_h) \notag \\
\leq&\ C(\bm\beta,K)  B_{2} (f_h), \label{eq:egstep}
\end{align} 
where in the last inequality we used \eqref{ass:main}, or that $\Sigma = (\beta_{1} + \cdots + \beta_{K})\delta/2$. Using the latter equation, we simplify the upper bounds in \eqref{eq:v1}--\eqref{eq:v4} as follows:
\begin{align*}
\abs{\tilde \epsilon_{i}(\tx)} \leq&\ C(\bm\beta,K)  \big(\delta B_{2}(f_h) +  B_{4}(f_h) \big),  \\
\abs{\tilde \epsilon_{ij}(\tx)} \leq&\ C(\bm\beta,K)  (  B_{3}(f_h) + \delta^{-1} B_{4}(f_h)),   \\
\abs{\tilde \epsilon_{ijk}(\tx)} \leq&\ C(\bm\beta,K)   \delta^{-1} B_{3}(f_h), \\
\abs{A \epsilon(\tx)} \leq&\ C(\bm\beta,K)  \delta^{-2} B_{4}(f_h) + C(\bm\beta,K) B_{1} (f_h)  N^{8K+1} \Big( 1- \frac{10K(\sqrt{N}-1)}{N} (1-\Sigma)\Big)^{N}.
\end{align*} 
Since $N^{8K+1} \Big( 1- \frac{10K(\sqrt{N}-1)}{N} (1-\Sigma)\Big)^{N}$ behaves approximately like $N^{8K+1} e^{-K\sqrt{N}}$ for large $N$, we can bound this term by, say, $C(\bm\beta,K) \delta^{5}$. Combining all of these inequalities   with \eqref{eq:egstep},  we conclude that for any $\tx \in \bs_{N}$, 
\begin{align*}
\abs{\epsilon_{\cG}(\tx)} \leq&\ C(\bm\beta,K) \big(\delta^{5} B_{1}(f_h)+ B_{2}(f_h) + \delta^{-1} B_{3}(f_h) + \delta^{-2} B_{4}(f_h) \big).
\end{align*}
Though the term $\delta^{5}$ in front of $B_{1}(f_h)$ could have been made smaller by choosing a larger exponent for $\delta$, there is no need for this, because $\delta^{5} B_{1}(f_h)$ is not a bottleneck error term. 
\end{proof}

%
\label{app:ingredientsagx}
We first state and prove an auxiliary lemma, followed by the proof of Lemma~\ref{lem:agxingredients}.
\begin{lemma}
\label{lem:binomialprob}
For any $\tu \in \s$ and any integer $M \geq 0$, 
\begin{align*}
\IP_{\tu}(U_{K}(1)  \leq M/N) \leq \frac{(M+1)N^{M}}{( \Sigma - p_{K})^{M}} ( 1- u_{K}(1-\Sigma))^{N}.
\end{align*}
\end{lemma}
\begin{proof}[Proof of Lemma~\ref{lem:binomialprob}]
Recall that from~\eq{eq:tp}, $N U_{K}(1)|u_{K} \sim$Binomial$(N, u_{K} - \bar u_{K})$, where $\bar u_{K} = -\Sigma u_{K} + p_{K}$. For any $0 \leq j \leq M$, 
\begin{align*}
\IP_{\tu}(U_{K}(1) = j/N) =&\ \frac{N!}{j!(N-j)!} (u_{K}(1-\Sigma) + p_{K})^{j} ( 1- u_{K}(1-\Sigma) - p_{K})^{N-j} \\
\leq&\    N^{j}  ( 1- u_{K}(1-\Sigma) - p_{K})^{N} \frac{1}{( 1- u_{K}(1-\Sigma) - p_{K})^{j}}\\ 
\leq&\ N^{M} ( 1- u_{K}(1-\Sigma))^{N} \frac{1}{( \Sigma - p_{K})^{j}}\\
\leq&\   \frac{N^{M}}{( \Sigma - p_{K})^{M}} ( 1- u_{K}(1-\Sigma))^{N},
\end{align*}
where the second inequality follows from $1- u_{K}(1-\Sigma) - p_{K} \geq \Sigma - p_{K} $ since $u_{K} \leq 1$. Thus, 
\begin{align*}
\IP_{\tu}(U_{K}(1)  \leq M/N) = \sum_{j=0}^{M} \IP_{\tu}(U_{K}(1) = j/N) \leq \frac{(M+1)N^{M}}{( \Sigma - p_{K})^{M}} ( 1- u_{K}(1-\Sigma) )^{N}.
\end{align*}
\end{proof}
\begin{proof}[Proof of Lemma~\ref{lem:agxingredients}]
Let us prove \eqref{eq:v1}. Recall that $A b_i(\tx) = b_i(\tx)$ by Corollary~\ref{cor:polynomials}, and note that from Lemma~ \ref{lem:productrule},
\begin{align*}
A(b_{i}\partial_{i} A f_h)(\tx) =&\  A b_{i}(x) A (\partial_{i} A f_{h}) (\tx)  + \epsilon_{i}(\tx) \\
=&\  b_{i}(\tx)   \partial_{i} A f_{h}  (\tx)  + \epsilon_{i}(\tx) +  b_{i}(\tx)  \big(A (\partial_{i} A f_{h}) (\tx) - \partial_{i} A f_{h}(\tx) \big),
\end{align*}
where,
\begin{align*}
\abs{\epsilon_{i}(\tx)} \leq  C(K)\max_{\substack{\norm{\ta}_{1}=1 \\ 0 \leq \tj \leq 4\te -\ta}} \abs{\Delta^{\ta} b_i(\delta (k(\tx) + \tj))}  \max_{\substack{\norm{\ta}_{1}=1 \\ 0 \leq \tj \leq 4\te -\ta}} \abs{\Delta^{\ta} \partial_{i} A f_h (\delta (k(\tx) + \tj))}.
\end{align*}
The mean value theorem and the expression for $b_i(\tu)$ in Lemma~\ref{lem:moments} implies that 
\begin{align*}
\max_{\substack{\norm{\ta}_{1}=1 \\ 0 \leq \tj \leq 4\te -\ta}} \abs{\Delta^{\ta} b_i(\delta (k\tx) + \tj))} \leq \delta \Sigma. 
\end{align*}
Similarly, 
\begin{align*}
\max_{\substack{\norm{\ta}_{1}=1 \\ 0 \leq \tj \leq 4\te -\ta}} \abs{\Delta^{\ta} \partial_{i} A f_h (\delta (k(\tx) + \tj))} \leq&\ \sup_{\substack{\norm{\ta}_{1}=2 \\ 0  \leq \tz \leq 5\delta \te }} \abs{\delta D^{\ta}  A f_h (\delta k(\tx) + \tz)} \\
\leq&\ C(K) \max_{\substack{\norm{\ta}_{1}=2 \\ 0  \leq \tj \leq  9\te - \ta}} \abs{\delta \delta^{-2} \Delta^{\ta} f_h(\delta(k(\tx)+\tj))} \\
\leq&\ C(K) \delta^{-1}B_{2}(f_h),
\end{align*}
where the second-last inequality is due to \eqref{eq:multibound2} of Theorem~\ref{thm:interpolant_def}. Thus, 
\begin{align*}
\abs{\epsilon_{i}(\tx)} \leq C(K) \Sigma B_{2}(f_h).
\end{align*}
To complete the proof of \eqref{eq:v1}, note that 
\begin{align*}
\abs{ b_{i}(\tx)  \big(A (\partial_{i} A f_{h}) (\tx) - \partial_{i} A f_{h}(\tx) \big)} \leq&\ \Sigma \abs{A (\partial_{i} A f_{h}) (\tx) - \partial_{i} A f_{h}(x)} \\
\leq&\ C(K) \Sigma \delta^{3} \max_{\norm{\ta}_{1}=4} \sup_{0 \leq \tz \leq  4\delta \te} \abs{D^{\ta} A f_h(\delta k(\tx) + \tz)} \\
\leq&\  C(K) \Sigma \delta^{-1} \max_{\substack{\norm{\ta}_{1}=4 \\ 0 \leq \tj \leq 8\te - \ta} }  \abs{\Delta^{\ta}  f_h(\delta(k(\tx) + \tj))} \\
\leq&\ C(K) \Sigma \delta^{-1} B_{4}(f_h).
\end{align*}
The proof of \eqref{eq:v2} is similar to that of \eqref{eq:v1}. Lemma~\ref{lem:productrule} and $A a_{ij}(\tx) = a_{ij}(\tx)$ implies that 
\begin{align*}
A(a_{ij}\partial_{i}\partial_{j} A f_h)(\tx) =&\  A a_{ij}(\tx) A (\partial_{i}\partial_{j} A f_{h}) (x)  + \epsilon_{ij}(\tx) \\
=&\   a_{ij}(x)   \partial_{i}\partial_{j} A f_{h}  (\tx)  + \epsilon_{ij}(\tx) +  a_{ij}(\tx)  \big(A (\partial_{i}\partial_{j} A f_{h}) (\tx) - \partial_{i}\partial_{j} A f_{h}(\tx) \big),
\end{align*}
where 
\begin{align*}
\abs{\epsilon_{ij}(\tx)} \leq  C(K)\max_{\substack{\norm{\ta}_{1}=1 \\ 0 \leq \tj \leq 4\te -\ta}} \abs{\Delta^{\ta} a_{ij}(\delta (k(\tx) + \tj))}  \max_{\substack{\norm{\ta}_{1}=1 \\ 0 \leq \tj \leq 4\te -\ta}} \abs{\Delta^{\ta} \partial_{i} \partial_{j} A f_h (\delta (k(\tx) + \tj))}.
\end{align*}
Note from Lemma~\ref{lem:moments} that  $B_{1}(a_{ij}) \leq C(\Sigma^{2} + \delta)\delta$. Thus, we can repeat the arguments used to bound $\epsilon_{i}(\tx)$ to see that
\begin{align*}
\abs{\epsilon_{ij}(\tx)} \leq&\ C(K)(\Sigma^{2} + \delta)\delta \max_{\substack{\norm{\ta}_{1}=1 \\ 0 \leq \tj \leq 4\te -\ta}} \abs{\Delta^{\ta} \partial_{i} \partial_{j} A f_h (\delta (k(\tx) + \tj))} \\
\leq&\ C(K)(\Sigma^{2} + \delta)\delta \delta^{-3} B_{3}(f_h).
\end{align*}
Furthermore, since  $\abs{a_{ij}(\tx)} \leq C(\Sigma^{2} + \delta)$, 
\begin{align*}
 \abs{ a_{ij}(\tx)  \big(A (\partial_{i}\partial_{j} A f_{h}) (\tx) - \partial_{i}\partial_{j} A f_{h}(\tx) \big)} \leq&\ C(\Sigma^{2} + \delta) \abs{A (\partial_{i}\partial_{j} A f_{h}) (\tx) - \partial_{i}\partial_{j} A f_{h}(\tx)} \\
\leq&\ C(K) (\Sigma^{2} + \delta) \delta^{2} \max_{\norm{\ta}_{1}=4} \sup_{0 \leq \tz \leq  4\delta \te} \abs{D^{\ta} A f_h(\delta k(\tx) + \tz)} \\
\leq&\  C(K) (\Sigma^{2} + \delta) \delta^{2}\delta^{-4} \max_{\substack{\norm{\ta}_{1}=4 \\ 0 \leq \tj \leq 8\te} }  \abs{\Delta^{\ta}  f_h(\delta(k(\tx) + \tj))} \\
\leq&\ C(K) (\Sigma^{2} + \delta) \delta^{-2} B_{4}(f_h),
\end{align*}
which proves \eqref{eq:v2}. Let us prove \eqref{eq:v3}. Note that 
\begin{align*}
&\abs{A ( c_{ijk}\partial_{i}\partial_{j}\partial_{k} A f_h)(\tx)} \\
 =&\ \bigg|\sum_{i_1, \ldots, i_{K-1} = 0}^{4} \bigg(\prod_{j=1}^{K-1} \alpha_{k_j(\tx) +i_j}^{k_j(\tx)   }(x_j)\bigg) c_{ijk}(\delta(k(\tx)+\ti)) \partial_{i}\partial_{j}\partial_{k} A f_h(\delta(k(\tx)+\ti)) \bigg| \\
 \leq&\ C(K)\left( \frac{1}{N} + \Sigma\right)^2 \sum_{i_1, \ldots, i_{K-1} = 0}^{4}\abs{\partial_{i}\partial_{j}\partial_{k} A f_h(\delta(k(\tx)+\ti))}\\
 \leq&\ C(K)\left( \frac{1}{N} + \Sigma\right)^2 \delta^{-3} \max_{\substack{\norm{\ta}_{1}=3 \\ 0 \leq \tj \leq 4\te -\ta}}\abs{\Delta^{\ta} f_h(\delta (k(\tx) + \tj))}\\
 \leq&\ C(K)\left( \frac{1}{N} + \Sigma\right)^2 \delta^{-3} B_{3}(f_h),
\end{align*}
where the first inequality is due to Corollary~\ref{cor:boundweights} and Lemma~\ref{lem:moments}, and the second inequality is due to \eqref{eq:multibound2} of Theorem~\ref{thm:interpolant_def}.
Lastly, we prove \eqref{eq:v4}. Recall that 
\begin{align*}
\epsilon(\tu) =&\ \frac{1}{6} \sum_{i,j,k=1}^{K-1}\E_{\tu}  \Big( (U_i(1) - u_i)(U_j(1) - u_j)(U_{k}(1)-u_{k}) \big(\partial_{i}\partial_{j}\partial_{k}A f_h(\bm{\xi}) - \partial_{i}\partial_{j}\partial_{k}A f_h(\tu)\big)\Big)  
\end{align*}
where $\bm{\xi}$ is between $\tu$ and $\tU(1)$. We will shortly prove that 
\begin{align}
\abs{\epsilon(\tu)} \leq&\ C(K) \delta^{-4} B_{4}(f_h) \max_{1 \leq i,j,k,\ell \leq K-1}\bar d_{ijk\ell}(\tu)  + C(K) N B_{1}(f_h) \IP_{u}(U_{K}(1) \leq 4K/N). \label{eq:epsintermbound}
\end{align}
Together with the bounds on $\bar d_{ijk\ell}(\tu)$ and $\IP_{\tu}(U_{K}(1) \leq 4K/N)$ in Lemmas~\ref{lem:moments} and \ref{lem:binomialprob}, respectively, we conclude that 
\begin{align*}
\abs{\epsilon(\tu)} \leq&\ C(K) \delta^{-4}\Big(\frac{2}{\sqrt{N}} + \Sigma\Big)^{4}  B_{4}(f_h) + C(K)N B_{1}(f_h) \frac{(4K+1)N^{4K}}{( \Sigma - p_{K})^{4K}} ( 1- u_{K}(1-\Sigma))^{N}.
\end{align*} 
Note that for any $\tx \in \bs_{N}$, 
\begin{align*}
\abs{A \epsilon(\tx)} =&\ \bigg| \sum_{i_1, \ldots, i_{K-1} = 0}^{4} \bigg(\prod_{j=1}^{d} \alpha_{k_j(\tx) +i_j}^{k_j(\tx)   }(x_j)\bigg) \epsilon(\delta (k(\tx) + \ti)) \bigg| \leq  C  \max_{0 \leq \ti \leq 4\te} \abs{\epsilon(\delta (k(\tx) + \ti))},
\end{align*}
where the inequality  is due to Corollary~\ref{cor:boundweights}. Since $\tx \in \bs_{N}$ implies that for any $0 \leq \ti \leq 4\te$, 
\begin{align*}
1 - \sum_{j=1}^{K-1}\delta(k_j(\tx)+i_j) \geq   1 - \sum_{j=1}^{K-1}\delta k_j(\tx)  - 4K/N \geq 10K\sqrt{N}/N - 4K/N \geq 10K(\sqrt{N}-1)/N,
\end{align*}
$\abs{A \epsilon(\tx)}$ is therefore bounded by
\begin{align*}
C(K) \delta^{-4}\Big(\frac{2}{\sqrt{N}} + \Sigma\Big)^{4}  B_{4}(f_h) + C(K) B_{1} (f_h)\frac{(4K+1) N^{4K+1}}{( \Sigma - p_{K})^{4K}} \Big( 1- \frac{10K(\sqrt{N}-1)}{N} (1-\Sigma)\Big)^{N},
\end{align*}
proving \eqref{eq:v4}.  We now prove \eqref{eq:epsintermbound}. By Lemma~\ref{lem:lipschitz}, $\partial_{i}\partial_{j}\partial_{k}A f_h \in \mathcal{D}^{K-1}$, implying that 
\begin{align*}
\abs{\partial_{i}\partial_{j}\partial_{k}A f_h(\bm{\xi}) - \partial_{i}\partial_{j}\partial_{k}A f_h(\tu)} \leq&\ \norm{\tU(1)-\tu}_{1} \sup_{\substack{ \min(\tu,\tU(1)) \leq \tz \\ \tz \leq \max(\tu,\tU(1)) \\ \norm{\ta}_{1} = 4}} \abs{D^{\ta} Af_h(\tz)}.
\end{align*}
Observe that
\begin{align*}
 &\sup_{\substack{ \min(\tu,\tU(1)) \leq \tz \\ \tz \leq \max(\tu,\tU(1)) \\ \norm{\ta}_{1} = 4}} \abs{D^{\ta} Af_h(\tz)} \\
  \leq&\ C(K) \delta^{-4}\sup_{\substack{ \min(\tu,\tU(1)) \leq \tz \\ \tz \leq \max(\tu,\tU(1)) \\ \norm{\ta}_{1} = 4}} \max_{0 \leq \ti \leq 4\te - \ta} \abs{\Delta^{\ta}  f_h(\delta(k(\tz) + \ti))} \\
 \leq&\  C(K) 1(U_{K}(1) > 4K/N)\delta^{-4}B_{4}(f_h)  \\
 & \hspace{2cm} +  C(K) 1(U_{K}(1) \leq 4K/N) \delta^{-4} \sup_{\substack{ \min(\tu,\tU(1)) \leq \tz \\ \tz \leq \max(\tu,\tU(1)) \\ \norm{\ta}_{1} = 4}} \max_{0 \leq \ti \leq 4\te - \ta} \abs{\Delta^{\ta}  f_h(\delta(k(\tz) + \ti))},
\end{align*}
where the first inequality follows from \eqref{eq:multibound2} of Theorem~\ref{thm:interpolant_def} and the second inequality follows from the fact that if $U_{K}(1) > 4K/N$, then   $\delta(k(\tz) + \ti) \in \s$ for all values of $\tz$ and $\ti$ considered  in the second line. If $U_{K}(1) \leq 4K/N$, then $\Delta^{a}  f_h(\delta(k(\tz) + \ti))$ may depend on values of $f_h(\tu)$ outside of $\s$, which is why we cannot bound it by $B_{4}(f_h)$. Instead, we observe that 
\begin{align*}
&1(U_{K}(1) \leq 4K/N) \sup_{\substack{ \min(\tu,\tU(1)) \leq \tz \\ \tz \leq \max(\tu,\tU(1)) \\ \norm{\ta}_{1} = 4}} \max_{0 \leq \ti \leq 4\te - \ta} \abs{\Delta^{\ta}  f_h(\delta(k(\tz) + \ti))} \\
\leq&\ 1(U_{K}(1) \leq 4K/N) C(K)\norm{f_h}.
\end{align*} 
To conclude \eqref{eq:epsintermbound}, we now show that 
\begin{align}
\norm{f_h} \leq K N B_1(f_h) = K \delta^{-1} B_1(f_h). \label{eq:b1bound}
\end{align}
Since we chose $f_h(\tu) = 0$ for $\tu$ outside of $\s$, we need to show that 
\begin{align*}
\abs{f_h(\tu)} \leq N K B_{1}(f_h),  \quad \tu \in \s,
\end{align*} 
Letting $\pi(\tu) = \IP(\tU = \tu)$, we have
\begin{align*}
f_h(\tu) = \sum_{n=0}^{\infty} \big( \E_{u} h(\tU(n)) - \E h(\tU) \big) =&\ \sum_{n=0}^{\infty} \sum_{\tu' \in \s} \pi(\tu') \big( \E_{\tu} h(\tU(n)) - \E_{\tu'} h(\tU(n)) \big) \\
=&\  \sum_{\tu' \in \s} \sum_{n=0}^{\infty}\pi(\tu') \big( \E_{\tu} h(\tU(n)) - \E_{\tu'} h(\tU(n)) \big),
\end{align*}
where the interchange is justified by the Fubini-Tonelli theorem because $\{\tU(n)\}$ is geometrically ergodic. Since 
\begin{align*}
\Delta_{i} f_h(\tu) = \sum_{n=0}^{\infty} \big( \E_{\tu + \delta \te_{i}} h(U(n)) - \E_{\tu} h(\tU(n)) \big),
\end{align*}
it follows that 
\begin{align*}
\sum_{n=0}^{\infty} \abs{\E_{\tu} h(\tU(n)) - \E_{\tu'} h(\tU(n)) } \leq  \norm{(\tu - \tu')/\delta}_{1} B_{1}(f_h)  \leq  N K B_{1}(f_h),
\end{align*} 
proving \eqref{eq:b1bound}. 
\end{proof}
 
%
%




We now prove Lemma~\ref{lem:moments}.
Recall that 
\begin{align*}
\Big( \big(NU_{1}(1),\ldots, NU_{K}(1)\big) \big| U(0) = \tu \Big) \sim \text{Multinomial}\Big(N, \big(u_{1} - \bar u_{1}, \ldots,u_{K} - \bar u_{K}\big) \Big),
\end{align*}
where $\bar u_{i} = u_{i}\Sigma  - p_{i}$. We also require the following result about the moments of the multinomial distribution. 
\begin{lemma}
\label{lem:multimoments}
Let $X \sim \text{Multinomial}(N, (p_1, \ldots, p_k))$. Then for all $i \neq j \neq k$,
\begin{align*}
\E(X_i) &= N p_i,\\
\E(X_i^2) &= N(N-1)p_i^2 + Np_i,\\
\E(X_i X_j) &= N(N-1)p_ip_j,\\
\E(X_iX_jX_k) &= N(N-1)(N-2)p_ip_jp_k,\\
\E(X_i^2X_j) &= N(N-1)(N-2) p_i^2p_j + N(N-1)p_ip_j,\\
\E(X_i^3) &= N(N-1)(N-2)p_i^3 + 3N(N-1)p_i^2 - 2Np_i.
\end{align*}
\end{lemma} 
\begin{proof}[Proof of Lemma~\ref{lem:moments}]
For convenience, we   write $\E(\cdot)$ instead of $\E_{u}(\cdot)$ and $U_{i}$ instead of $U_{i}(1)$; e.g., we write $\E U_{i}$ instead of $\E_{u} U_{i}(1)$.  Since $\E N U_i = N(u_{i} - \bar u_{i})$, it follows from Lemma~\ref{lem:multimoments} that 
\begin{align*}
b_{i}(\tu) = \E(U_i - u_i) = -\bar u_{i}.
\end{align*}
Next, we observe that
\begin{align*}
a_{ij}(\tu) =  \E(U_i - u_i)(U_j - u_j)  =&\   \E( U_{i} U_{j}) - u_{i}  \E   U_{j}- u_{j}  \E   U_{i} + u_{i} u_{j}.
\end{align*}
Lemma~\ref{lem:multimoments} implies that $\E U_{j} = (u_{j} - \bar u_{j})$ and
\begin{align*}
  \E ( U_{i} U_{j}) =&\   (1-1/N)(u_i - \bar u_i) (u_j - \bar u_{j}) \\
=&\ u_i u_j - u_i \bar u_j - u_j \bar u_i + \bar u_i \bar u_j- \frac{1}{N}(u_i - \bar u_i) (u_j - \bar u_{j}),
\end{align*}
from which it follows that
\begin{align}
a_{ij}(\tu) =&\  \bar u_i \bar u_j-\frac{1}{N}(u_i - \bar u_i) (u_j - \bar u_{j}). \label{eq:aijproof}
\end{align}
Similarly,
\begin{align*}
 a_{ii}(\tu) = \E(U_i - u_i)^2 =&\   \E( U_{i}^2 ) - 2 u_{i} \E  U_{i}  + u_{i}^2 = \E( U_{i}^2 ) - 2 u_{i} (u_{i} - \bar u_{i})  + u_{i}^2.
\end{align*}
By Lemma~\ref{lem:multimoments}, we know that 
\begin{align*}
 \E ( U_{i}^2  ) =&\   (1-1/N)(u_i - \bar u_i)^2 + \frac{1}{N} (u_i-\bar u_i) \\
=&\ (u_i^2 - 2u_i \bar u_i + \bar u_i^2) + \frac{1}{N} \big( (u_i - \bar u_i) - (u_i - \bar u_i)^2\big),
\end{align*}
and we conclude that 
\begin{align}
a_{ii}(u) =&\ \bar u_i^2 + \frac{1}{N} \big((u_i - \bar u_i) - (u_i - \bar u_i)^2\big). \label{eq:aiiproof}
\end{align}
To prove $\abs{c_{ijk}(\tu)} \leq  C\left( \frac{1}{N} + \Sigma\right)^2$, we compute and then bound $c_{iii}(\tu)$, $c_{iij}(\tu)$, and $c_{ijk}(\tu)$, for $i \neq j \neq k$.  We begin with
\begin{align*}
c_{iii}(\tu) =&\ \E (U_i-u_i)^{3} =  \E U_{i}(U_{i}-u_{i})^{2} - u_{i} \E(U_i - u_i)^2  \\
=&\ \E(U_{i}^{3} - 2 u_{i} U_{i}^{2} + u_{i}^{2} U_{i}) - u_{i} \E(U_i - u_i)^2. 
\end{align*}
From \eqref{eq:aiiproof}, we know that 
\begin{align*} 
- u_{i} \E(U_i - u_i)^2 =&\ -u_{i}   \bar u_i^2 - \frac{1}{N} u_{i}\big((u_i - \bar u_i) - (u_i - \bar u_i)^2  \big).
\end{align*}
For the remaining terms, we use Lemma~\ref{lem:multimoments} to see that 
\begin{align*}
\E U_{i}^{3} =&\ (1-1/N)(1-2/N)(u_i - \bar u_i )^3 + \frac{3}{N} (1-1/N)(u_i - \bar u_i )^2 -  \frac{2}{N^2} (u_i - \bar u_i) \\
=&\ (u_i - \bar u_i )^3  - \frac{3}{N}(u_i - \bar u_i )^3 + \frac{2}{N^2}(u_i - \bar u_i )^3 + \frac{3}{N} (u_i - \bar u_i )^2 - \frac{3}{N^2}(u_i - \bar u_i )^2 - \frac{2}{N^2}(u_i - \bar u_i ),\\
-2u_i \E U_{i}^2 =&\ -2u_i \Big( (1-1/N) (u_i - \bar u_i )^2 + \frac{1}{N} (u_i - \bar u_i ) \Big) \\
=&\ -2 u_i (u_i - \bar u_i )^2 + 2 u_i \frac{1}{N} (u_i - \bar u_i )^2 - 2u_i \frac{1}{N} (u_i - \bar u_i ), \\
u_i^2 \E U_i =&\ u_i^2 (u_i - \bar u_i).
\end{align*}
Since $\abs{u_{i}} \leq 1$ and $\abs{\bar u_{i}} \leq \Sigma$, all terms with $1/N^2$ in front of them can be ignored, because they are trivially bounded by $C/N^2$. Of the remaining terms, let us consider only those containing $1/N$ in front of them. Namely, 
\begin{align*}
& - \frac{3}{N}(u_i - \bar u_i )^3+ \frac{3}{N} (u_i - \bar u_i )^2+ 2 u_i \frac{1}{N} (u_i - \bar u_i )^2 - 2u_i \frac{1}{N} (u_i - \bar u_i )- \frac{1}{N} u_{i}\big((u_i - \bar u_i) - (u_i - \bar u_i)^2\big)   \\
=&\ \frac{3}{N} \Big( - (u_i - \bar u_i )^3+  (u_i - \bar u_i )^2+   u_i  (u_i - \bar u_i )^2 -  u_i   (u_i - \bar u_i ) \Big) \\
=&\ \frac{3}{N} \Big(  \bar u_i (u_i - \bar u_i )^2 +  (u_i - \bar u_i )^2  -  u_i   (u_i - \bar u_i ) \Big) \\
=&\ \frac{3}{N} \Big(  \bar u_i (u_i - \bar u_i )^2  - \bar u_i (u_i - \bar u_i ) \Big) \\
=&\ \frac{3}{N} \bar u_i (u_i - \bar u_i ) \Big(    (u_i - \bar u_i )   -  1\Big),
\end{align*}
and note that 
\begin{align*}
\abs{\frac{3}{N} \bar u_i (u_i - \bar u_i ) \Big(    (u_i - \bar u_i )   -  1\Big)}  \leq C \frac{ \Sigma}{N}.
\end{align*}
Lastly, we collect all the terms without $1/N$ or $1/N^2$ in front of them. Their sum equals
\begin{align*}
&(u_i - \bar u_i )^3  -2 u_i (u_i - \bar u_i )^2 + u_i^2 (u_i - \bar u_i)-u_{i}   \bar u_i^2 \\
=&\ u_i^3 - 3\bar u_i u_i^2 + 3 \bar u_i^2 u_i - \bar u_i^3 - 2u_i^3 + 4 u_i^2 \bar u_i - 2 u_i \bar u_i^2 + u_i^3 - u_i^2 \bar u_i - u_i \bar u_i^2\\
=&\ - 3\bar u_i u_i^2 + 3 \bar u_i^2 u_i - \bar u_i^3 + 4 u_i^2 \bar u_i - 2 u_i \bar u_i^2  - u_i^2 \bar u_i - u_i \bar u_i^2\\
=&\  3 \bar u_i^2 u_i - \bar u_i^3  - 2 u_i \bar u_i^2  - u_i \bar u_i^2 \\
=&\ - \bar u_i^3.
\end{align*}
Since $\abs{\bar u_{i}^{3}} \leq \Sigma^3 \leq \Sigma^2$, we have shown that $\abs{c_{iii}(u)} \leq C(1/N + \Sigma)^{2}$. The remaining bounds are shown similarly. Next, we bound
\begin{align*}
c_{iij}(\tu) =&\  \E(U_i-u_i)^2(U_j-u_j) =  \E U_i(U_i-u_i)(U_j-u_j) - u_i\E(U_i-u_i)(U_j-u_j) \\
=&\ \E U_{i}^{2} U_{j}  - u_{j}  \E U_{i}^2  - u_{i}\E U_{i}U_{j} + u_{i}u_{j} \E U_{i} - u_i\E(U_i-u_i)(U_j-u_j).
\end{align*}
From \eqref{eq:aijproof}, we know that 
\begin{align*}
- u_i\E(U_i-u_i)(U_j-u_j) =&\ -u_i \Big( \bar u_i \bar u_j- \frac{1}{N} (u_i - \bar u_i) (u_j - \bar u_{j}) \Big),
\end{align*}
and for the rest of the terms we use Lemma~\ref{lem:multimoments} to get 
\begin{align*}
\E U_{i}^{2} U_{j} =&\ (1-1/N)(1-2/N) (u_i - \bar u_i)^2  (u_j - \bar u_j) + \frac{1}{N}(1-1/N)  (u_i - \bar u_i)  (u_j - \bar u_j) \\
=&\  (u_i - \bar u_i)^2  (u_j - \bar u_j) - \frac{3}{N} (u_i - \bar u_i)^2  (u_j - \bar u_j)+ \frac{2}{N^2} (u_i - \bar u_i)^2  (u_j - \bar u_j)  \\
&+ \frac{1}{N} (u_i - \bar u_i)  (u_j - \bar u_j) - \frac{1}{N^2} (u_i - \bar u_i)  (u_j - \bar u_j), \\
 - u_{j}  \E U_{i}^2 =&\ -u_{j} \Big( (1-1/N)  (u_i - \bar u_i)^2 + \frac{1}{N}  (u_i - \bar u_i) \Big) \\
 =&\ -u_{j} \Big(  (u_i - \bar u_i)^2 - \frac{1}{N } (u_i - \bar u_i)^2 +   \frac{1}{N} (u_i - \bar u_i)  \Big), \\
- u_{i}\E U_{i}U_{j} =&\ - u_{i} \Big( (u_i - \bar u_i)  (u_j - \bar u_j) - \frac{1}{N}(u_i - \bar u_i)  (u_j - \bar u_j) \Big), \\
u_{i}u_{j} \E U_{i} =&\   u_{i}u_{j}(u_i - \bar u_i).
\end{align*}
We again ignore all terms with $1/N^2$ in front because they can be bounded by $C/N^2$. Collecting all terms with $1/N$ in front, the result equals 
\begin{align*}
& \frac{1}{N} \Big(- 3 (u_i - \bar u_i)^2  (u_j - \bar u_j)+   (u_i - \bar u_i)  (u_j - \bar u_j) + u_j (u_i - \bar u_i)^2 - u_j (u_i - \bar u_i) \\
& \quad  + u_{i} (u_i - \bar u_i)  (u_j - \bar u_j) + u_{i} (u_i - \bar u_i)  (u_j - \bar u_j) \Big) \\
=&\  \frac{1}{N}(u_i - \bar u_i)  \Big(- 3 (u_i - \bar u_i)  (u_j - \bar u_j)+     (u_j - \bar u_j) + u_j (u_i - \bar u_i) - u_j + 2 u_{i}   (u_j - \bar u_j)   \Big)\\
=&\  \frac{1}{N}(u_i - \bar u_i)  \Big(- 3 u_i u_j + 3 u_i \bar u_j + 3 u_j \bar u_i - 3 \bar u_i \bar u_j   + (u_j - \bar u_j) +  u_j  u_i   - u_j  \bar u_i  - u_j + 2 u_{i} u_j  - 2 u_{i} \bar u_j    \Big)\\
=&\  \frac{1}{N}(u_i - \bar u_i)  \Big( 3 u_i \bar u_j + 3 u_j \bar u_i - 3 \bar u_i \bar u_j    - \bar u_j   - u_j  \bar u_i   - 2 u_{i} \bar u_j    \Big),
\end{align*}
the absolute value of which can be bounded by $C \Sigma/N$, because $\bar u_{i}, \bar u_{j} \leq \Sigma$. Next, collecting all terms without $1/N^2$ or $1/N$ yields  
\begin{align*}
&(u_i - \bar u_i)^2  (u_j - \bar u_j)- u_{j}  (u_i - \bar u_i)^2- u_{i}  (u_i - \bar u_i)  (u_j - \bar u_j) +u_{i}u_{j}(u_i - \bar u_i)-u_i   \bar u_i \bar u_j \\
=&\ - \bar u_j (u_i - \bar u_i)^2   - u_{i}  (u_i - \bar u_i)  (u_j - \bar u_j) +u_{i}u_{j}(u_i - \bar u_i)-u_i   \bar u_i \bar u_j\\
=&\ - \bar u_j (u_i - \bar u_i)^2  +u_{i}  \bar u_j (u_i - \bar u_i)  -u_i   \bar u_i \bar u_j\\
=&\ - \bar u_j \Big(  (u_i - \bar u_i)^2  - u_{i}   (u_i - \bar u_i)  +u_i   \bar u_i \Big)\\
=&\ - \bar u_j \Big(  - \bar u_i (u_i - \bar u_i)    +u_i   \bar u_i \Big)\\
=&\ - \bar u_j \bar u_i^2,
\end{align*}
which is bounded by $\Sigma^{3}$. Thus we have shown that $\abs{c_{iij}(\tu)} \leq C(1/N + \Sigma)^{2}$. Lastly, we consider  
\begin{align*}
c_{ijk}(\tu) =&\ \E(U_i-u_i)(U_j-u_j) (U_k-u_k) \\
=&\ \E U_{k} (U_i-u_i)(U_j-u_j)  - u_{k} \E (U_i-u_i)(U_j-u_j)  \\
=&\ \E U_{k} U_{i} U_{j} - u_{j} \E U_{k}  U_{i} - u_{i} \E U_{k} U_{j} + u_{i} u_{j} \E U_{k} - u_{k} \E (U_i-u_i)(U_j-u_j)
\end{align*}
From \eqref{eq:aijproof}, we have 
\begin{align*}
- u_{k} \E (U_i-u_i)(U_j-u_j) =&\ - u_{k}  \Big( \bar u_i \bar u_j- \frac{1}{N} (u_i - \bar u_i) (u_j - \bar u_{j}) \Big),
\end{align*}
and from Lemma~\ref{lem:multimoments}, we have 
\begin{align*}
\E U_{k} U_{i} U_{j} =&\ \Big( 1 - \frac{3}{N} + \frac{2}{N^2} \Big) (u_{i} - \bar u_{i})(u_{j} - \bar u_{j})(u_{k} - \bar u_{k}), \\
- u_{j} \E U_{k}  U_{i} =&\  - u_{j} \Big( (u_k - \bar u_k)  (u_i - \bar u_i) - \frac{1}{N}(u_k - \bar u_k)  (u_i - \bar u_i) \Big) \\
- u_{i} \E U_{k}  U_{j} =&\  - u_{i} \Big( (u_k - \bar u_k)  (u_j - \bar u_j) - \frac{1}{N}(u_k - \bar u_k)  (u_j - \bar u_j) \Big).
\end{align*}
Collecting terms with $1/N$ in front yields 
\begin{align*}
& \frac{1}{N} \Big(  -3(u_{i} - \bar u_{i})(u_{j} - \bar u_{j})(u_{k} - \bar u_{k}) +  u_{j}  (u_k - \bar u_k)  (u_i - \bar u_i) \\
& \qquad + u_{i}  (u_k - \bar u_k)  (u_j - \bar u_j) +  u_k (u_i - \bar u_i) (u_j - \bar u_{j}) \Big)  \\
=&\ \frac{1}{N} \Big( \bar u_j (u_{i} - \bar u_{i})(u_{k} - \bar u_{k}) + \bar u_i (u_{j} - \bar u_{j})(u_{k} - \bar u_{k})+ \bar u_k (u_{j} - \bar u_{j})(u_{i} - \bar u_{i}) \Big),
\end{align*}
and this term can be bounded by $C \Sigma /N$. Similarly, collecting all terms without $1/N^2$ and $1/N$ yields  
\begin{align*}
&(u_{i} - \bar u_{i})(u_{j} - \bar u_{j})(u_{k} - \bar u_{k}) - u_{j}  (u_k - \bar u_k)  (u_i - \bar u_i) - u_{i}  (u_k - \bar u_k)  (u_j - \bar u_j) + u_{i} u_{j} (u_k - \bar u_k)  -  u_{k}   \bar u_i \bar u_j \\
=&\  - \bar u_{j} (u_{i} - \bar u_{i}) (u_{k} - \bar u_{k})- u_{i}  (u_k - \bar u_k)  (u_j - \bar u_j) + u_{i} u_{j} (u_k - \bar u_k)  -  u_{k}   \bar u_i \bar u_j  \\
=&\ - \bar u_{j} ( u_{i} u_k - u_k \bar u_i - u_i \bar u_k + \bar u_i \bar u_k) - u_{i} (u_k u_j -u_k \bar u_j - u_j \bar u_k + \bar u_j \bar u_k)  + u_{i} u_{j} u_k - u_{i} u_{j}  \bar u_k   -  u_{k}   \bar u_i \bar u_j\\
=&\ - \bar u_{j} ( u_{i} u_k - u_k \bar u_i - u_i \bar u_k + \bar u_i \bar u_k) - u_{i} ( -u_k \bar u_j   + \bar u_j \bar u_k)       -  u_{k}   \bar u_i \bar u_j\\
=&\ - \bar u_{j} (   - u_k \bar u_i - u_i \bar u_k + \bar u_i \bar u_k) - u_{i}\bar u_j \bar u_k       -  u_{k}   \bar u_i \bar u_j,
\end{align*}
and this term can be bounded by $C \Sigma^2$. This completes the bound for $\abs{c_{ijk}(\tu)}$.  

Finally we show that
\begin{align*}
\bar d_{ijk\ell}(\tu) \leq \left( \frac{2}{\sqrt N} + \Sigma\right)^4. 
\end{align*}
Using the Cauchy-Schwarz inequality twice, 
\begin{align*}
&\E\abs{(U_i - u_i)(U_j - u_j)(U_k - u_k)(U_{\ell} - u_{\ell})} \\
\leq&\ \left[\E(U_i - u_i)^4\right]^\frac14 \left[\E_u(U_j - u_j)^4\right]^\frac14 \left[\E(U_k - u_k)^4\right]^\frac14 \left[\E_u(U_{\ell} - u_{\ell})^4\right]^\frac14.
\end{align*}
Recalling that $N U_i | u_i \sim \text{Bin}( N, u_i - \bar u_{i})$ and then using Minkowski's inequality,
\begin{align*}
\left[\E(U_i - u_i)^4\right]^\frac14 &= \frac{1}{N} \left[\E_u ( NU_i - N(u_i - \bar u_{i}) - N \bar u_{i})^4 \right]^\frac14\\
	& \leq \frac{1}{N} \left[ \E_u ( NU_i - N(u_i- \bar u_{i}))^4 \right]^\frac14 + \frac{1}{N}  N \bar u_{i}.
\end{align*}
Noting that for $Y \sim \text{Bin}(n,p)$,
\begin{align*}
\E(Y - np)^4 &= 3(np(1-p))^2 + np(1-p)(1-6p(1-p)) \\
	&\leq 3(np(1-p)^2 + np(1-p) \leq 4n^2,
\end{align*}
we conclude that 
\begin{align*}
\left[\E(U_i - u_i)^4\right]^\frac14 &\leq \frac{1}{N}\Big( (4N^2)^\frac14 + N \bar u_{i}\Big) \leq \frac{2}{\sqrt N} + \Sigma,
\end{align*}
and the final bound is now clear.
\end{proof}

Finally we complete this section with the proof of the Stein factor bounds in Lemma~\ref{lem:steinfactors}.

\begin{proof}[Proof of Lemma~\ref{lem:steinfactors}]
For notational clarity, given $\ta$ such that $\|\ta\| = k$, decompose $\ta$ into its components such that $\ta = \ta_1 + \ldots + \ta_k$, 
where $\ta_i$ is a standard basis unit vector. Furthermore, for the remainder of this proof, let $\tU_\tu(t)$ denote the process $\tU(t)$ started at $\tU(0) = \tu$. We first start with the case where $\| \ta \| = 1$. 
Starting from~\eq{eq:stnsol}, given $h \in \mathcal \cM_{disc,4}(C)$
\ban{
\abs{\Delta^{\ta_1} f_h(\tu)}  = \bbbabs{\sum_{t=0}^\infty \left[\E h (\tU_{\tu + \delta\ta_1}(t)) - \E h (\tU_\tu(t)) \right] } \leq \sum_{t=0}^\infty C \E  \| \tU_{\tu + \delta \ta_1 }(t) - \tU_{\tu}(t) \|_{1}. \label{eq:del1}
} 
We couple the two processes $\tU_{\tu + \delta \ta_1}(t)$ and $\tU_\tu(t)$ in the following manner. Index the parents so that the types for individuals in both processes match except for the one entry where the first process will have an individual of type depending on $\ta_1$ and the second process has an individual of type $K$. (Recall we reserve the final type $K$ to be the remainder.) Given the current generation, to generate the next generation, take a random sample of size $N$ from the indices $\{1, \ldots N\}$, and use this common random sample to choose parents for the offspring both processes. Given mutation is parent independent, we also couple the mutations identically across both processes in the obvious manner.  Figure~\ref{fig:1} illustrates the joint evolution of $\tU_{\tu + \delta \ta_1}(t)$ and $\tU_\tu(t)$. 
\begin{figure}[h!]
\centering 
\includegraphics[scale=1.2]{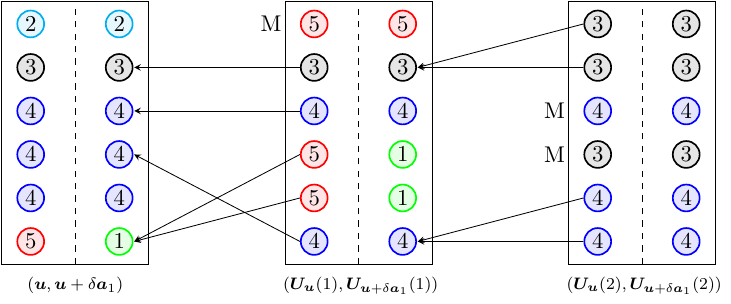}
\caption{An illustration of our coupling at times $t = 0,1,2$ with population   size $N = 6$  and $K= 5$ types. Each gene type is color and number coded; e.g., type $1$ is green, type $5$ is red, etc.  An ``M'' next to a row represents a mutation, while arrows represent parental relationships; e.g., rows two and three (from the bottom) in the middle plot are children of the first row in the leftmost plot, while row six of the middle plot mutated. Coupling occurs at time $t = 2$ since rows two and three in the middle plot have no children.   }
\label{fig:1}
\end{figure}


Given the starting configuration at time $0$, let $\tV(t) = (V_1(t), \ldots, V_N(t))$ denote the process that tracks the ancestry of the original configuration.That is $V_j(t)$ tracks the number of individuals at time $t$ that trace their ancestry directly back to individual $j$ at time $0$. Note that if a mutation occurs, it is removed from this process, hence $\|\tV(0)\|_1 = N$, and $\lim_{t \to \infty} \tV(t) = 0$.  For readers familiar with coalescent theory, this is analogous to a coalescent process looking forwards in time.

Without loss of generality, we can therefore set
\ba{ 
\| \tU_{\tu + \delta\ta_1 }(t) - \tU_{\tu}(t) \|_{1} = \delta V_1(t).
}
In the context of Figure~\ref{fig:1}, $V_1(t)$ is tracking the number of replicates at time $t$ of the pair $(5,1)$ from the final row in of the processes at time 0. Therefore in this particular realisation, $V_1(0) = 1, V_1(1) = 2, V_1(2) = 0$. Given $V_1(t-1)$, $V_1(t)$ will be made up of the number of times one of the $V_1(t-1)$ individuals are chosen, which will occur with probability $\delta V_1(t-1)$ who also do not mutate. Hence, recalling that $\Sigma = \sum_{i=1}^K p_k$ denotes the probability of any mutation occurring,
\ba{
V_1(t) | V_1(t-1) \sim \Bin\Big(N, \delta V_1(t-1) (1-\Sigma) \Big).
}
Observing that $\IE[V_1(1)] = \IE[ \IE[V_1(1) | V_1(0)]] = \delta\IE[V_1(0)](1-\Sigma) =  \delta(1-\Sigma)$, and then applying this recursively we conclude that $\IE [V_1(t)] = \delta(1-\Sigma)^t$ for all integers $t \geq 0$. Therefore, following on from~\eq{eq:del1}, for $h \in \mathcal M_{disc,4}(C)$
\ba{
\abs{\Delta^{\ta_1}f_h(\tu)} \leq \sum_{t=0}^\infty C \delta \E V_1(t) =\sum_{t=0}^\infty C \delta(1-\Sigma)^t= \frac{C\delta}{\Sigma}.
}
For the second-order difference where $\|\ta\|=2$, for $h \in \mathcal M_{disc,4}(C)$,
\ban{ 
\abs{\Delta^{\ta} f_h(\tu) }&= \bbbabs{ \sum_{t=0}^\infty \E \left[ h (\tU_{\tu + \d \ta_1 +\d \ta_2}(t)) - h (\tU_{\tu + \d \ta_1}(t)) - h (\tU_{\tu + \d \ta_2}(t)) + h (\tU_\tu(t)) \right] }\notag \\
	&\leq C \sum_{t=0}^\infty \E \left[\|  \tU_{\tu + \ta_1 + \ta_2}(t)- \tU_{\tu + \ta_1}(t)\|_1 \|  \tU_{\tu + \ta_2}(t) - \tU_{\tu}(t)\|_1\right], \label{eq:del2}
}
where the inequality is due to the fact that, in general,
\begin{align}
\abs{f(\tx+\ta+ \tb) - f(\tx+\ta) - f(\tx+\tb) - f(\tx)} \leq \|\ta\|_1\|\tb\|_1 \max_{\norm{\ta'}=2} \norm{\Delta^{\ta'}f}. \label{eq:ingeneral}
\end{align}
For one-dimensional functions $f: \Z \to \R$,  inequality \eqref{eq:ingeneral}  follows from  
\begin{align*}
&f(x + a + b) - f(x + a) -f(x+b) + f(x) \\
 =&\ \sum_{i=0}^{b-1} \sum_{j=1}^{a-1} \big(f(x +j+ i) - 2 f(x +j-1+ i)+ f(x +j-1+ i -1)\big), \quad x,a,b \in \Z.
\end{align*}
A similar idea can be used to justify \eqref{eq:ingeneral} for multidimensional functions.

We couple the 4 processes on the right-hand side in an analogous manner to the first difference; i.e, parent selection and mutation are coupled to be identical for all 4 processes.  Figure~\ref{fig:2} illustrates their evolution. Similar to the first difference, we need to keep track of any differences between the 4 processes, which can be achieved by tracking genealogies using the process $\tV(t)$ where we set $(V_1(t), V_2(t))$ to jointly track the the propagation of the initial two rows in Figure~\ref{fig:2}.

Therefore we can without loss of generality set
\ba{
\bigg( \|  \tU_{\tu + \ta_1 + \ta_2}(t)- \tU_{\tu + \ta_1}(t)\|_1, \|  \tU_{\tu + \ta_2}(t) - \tU_{\tu}(t)\|_1 \bigg) = \bigg(\delta V_1(t), \delta V_2(t)\bigg).
}
\begin{figure}[h!]
\centering 
\includegraphics[scale=1.2]{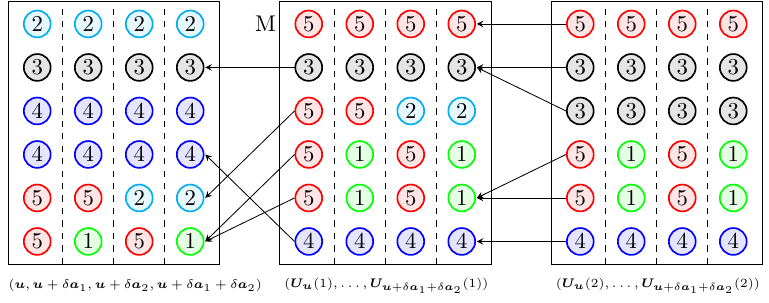}
\caption{An illustration of our second-order difference coupling at times $t = 0,1,2$ with population   size $N = 6$  and $K= 5$ types.    }
\label{fig:2}
\end{figure} 
Analogously to the first difference, we can $V_1(t)$ and $V_2(t)$ jointly such that
\ba{
\Big( (V_1(t), V_2(t)) \Big| V_1(t-1), V_2(t-1)\Big) \sim \text{Multinomial}\left(N, \left\{ \delta V_1(t-1) (1-\Sigma), \delta V_2(t-1) (1-\Sigma)\right\}\right).
}
Given $V_1(0) = V_2(0) = 1$, we can recursively show, with the help of  Lemma~\ref{lem:multimoments}, that $\IE[N^2 V_1(t) V_2(t)] = N(N-1)\delta^2(1-\Sigma)^{2t}$. Hence continuing from~\eq{eq:del2}
\ba{
 \abs{\Delta^{\ta} f_h(\tu)}\leq \sum_{t=0}^\infty C \delta^2\E V_1(t)V_2(t) \leq \sum_{t=0}^\infty C\delta^2 (1-\Sigma)^{2t}=  \frac{C\delta^2}{(1-(1-\Sigma)^2)}.
}
We omit the proof of the bounds when $\norm{\ta} = 3,4$, as they follow from exactly the same proof methodology.
\end{proof}

\bibliographystyle{apalike}
\bibliography{references}

\def\cprime{$'$} \def\cprime{$'$} \def\cprime{$'$} \def\cprime{$'$}
  \def\cprime{$'$} \def\cprime{$'$} \def\cprime{$'$}
\begin{thebibliography}{}

\bibitem[Barbour, 1988]{B88}
Barbour, A.~D. (1988).
\newblock Stein's method and {P}oisson process convergence.
\newblock {\em J. Appl. Probab.}, (Special Vol. 25A):175--184.
\newblock A celebration of applied probability.

\bibitem[Barbour and Chen, 2005]{introstein}
Barbour, A.~D. and Chen, L. H.~Y., editors (2005).
\newblock {\em An introduction to {S}tein's method}, volume~4 of {\em Lecture
  Notes Series. Institute for Mathematical Sciences. National University of
  Singapore}.
\newblock Singapore University Press, Singapore; World Scientific Publishing
  Co. Pte. Ltd., Hackensack, NJ.
\newblock Lectures from the Meeting on Stein's Method and Applications: a
  Program in Honor of Charles Stein held at the National University of
  Singapore, Singapore, July 28--August 31, 2003.

\bibitem[Barbour et~al., 1992]{BHJ}
Barbour, A.~D., Holst, L., and Janson, S. (1992).
\newblock {\em Poisson approximation}, volume~2 of {\em Oxford Studies in
  Probability}.
\newblock The Clarendon Press, Oxford University Press, New York.
\newblock Oxford Science Publications.

\bibitem[Braverman, 2022a]{B22}
Braverman, A. (2022a).
\newblock The prelimit generator comparison approach of stein’s method.
\newblock {\em Stochastic Systems}, 12(2):181--204.

\bibitem[Braverman, 2022b]{Brav2022}
Braverman, A. (2022b).
\newblock Stein factor bounds for the join the shortest queue model in the
  {H}alfin-{W}hitt regime.
\newblock Working paper.

\bibitem[Braverman, 2023]{Brav2023}
Braverman, A. (2023).
\newblock The join-the-shortest-queue system in the {H}alfin-{W}hitt regime:
  Rates of convergence to the diffusion limit.
\newblock {\em Stochastic Systems}, 13(1):1--39.

\bibitem[Brown and Phillips, 1999]{BP99}
Brown, T.~C. and Phillips, M.~J. (1999).
\newblock Negative binomial approximation with {S}tein's method.
\newblock {\em Methodol. Comput. Appl. Probab.}, 1(4):407--421.

\bibitem[Chatterjee, 2014]{Chatterjee2014}
Chatterjee, S. (2014).
\newblock A short survey of {S}tein's method.
\newblock In {\em Proceedings of the {I}nternational {C}ongress of
  {M}athematicians---{S}eoul 2014. {V}ol. {IV}}, pages 1--24. Kyung Moon Sa,
  Seoul.

\bibitem[Chen, 1975]{Chen75}
Chen, L. H.~Y. (1975).
\newblock Poisson approximation for dependent trials.
\newblock {\em Ann. Probability}, 3(3):534--545.

\bibitem[D\"{o}bler, 2015]{Dobler2015}
D\"{o}bler, C. (2015).
\newblock Stein's method of exchangeable pairs for the beta distribution and
  generalizations.
\newblock {\em Electron. J. Probab.}, 20:no. 109, 34.

\bibitem[Ethier and Norman, 1977]{EN77}
Ethier, S.~N. and Norman, M.~F. (1977).
\newblock Error estimate for the diffusion approximation of the wright--fisher
  model.
\newblock {\em Proceedings of the National Academy of Sciences},
  74(11):5096--5098.

\bibitem[Fulman and Ross, 2013]{FulmanRoss2013}
Fulman, J. and Ross, N. (2013).
\newblock Exponential approximation and {S}tein's method of exchangeable pairs.
\newblock {\em ALEA Lat. Am. J. Probab. Math. Stat.}, 10(1):1--13.

\bibitem[Gan et~al., 2017]{GRR17}
Gan, H.~L., R\"{o}llin, A., and Ross, N. (2017).
\newblock Dirichlet approximation of equilibrium distributions in {C}annings
  models with mutation.
\newblock {\em Adv. in Appl. Probab.}, 49(3):927--959.

\bibitem[Gan and Ross, 2021]{GR21}
Gan, H.~L. and Ross, N. (2021).
\newblock Stein's method for the {P}oisson-{D}irichlet distribution and the
  {E}wens sampling formula, with applications to {W}right-{F}isher models.
\newblock {\em Ann. Appl. Probab.}, 31(2):625--667.

\bibitem[Goldstein and Reinert, 2013]{GR13}
Goldstein, L. and Reinert, G. (2013).
\newblock Stein's method for the beta distribution and the
  {P}\'{o}lya-{E}ggenberger urn.
\newblock {\em J. Appl. Probab.}, 50(4):1187--1205.

\bibitem[Ley et~al., 2017]{LRS2017}
Ley, C., Reinert, G., and Swan, Y. (2017).
\newblock Stein’s method for comparison of univariate distributions.
\newblock {\em Probability Surveys}, 14:1--52.

\bibitem[Mackey and Gorham, 2016]{GM16}
Mackey, L. and Gorham, J. (2016).
\newblock Multivariate {S}tein factors for a class of strongly log-concave
  distributions.
\newblock {\em Electron. Commun. Probab.}, 21:14.

\bibitem[Ross, 2011]{Ross11}
Ross, N. (2011).
\newblock Fundamentals of {S}tein's method.
\newblock {\em Probab. Surv.}, 8:210--293.

\bibitem[Stein, 1972]{stein72}
Stein, C. (1972).
\newblock A bound for the error in the normal approximation to the distribution
  of a sum of dependent random variables.
\newblock In {\em Proceedings of the {S}ixth {B}erkeley {S}ymposium on
  {M}athematical {S}tatistics and {P}robability ({U}niv. {C}alifornia,
  {B}erkeley, {C}alif., 1970/1971), {V}ol. {II}: {P}robability theory}, pages
  583--602. Univ. California Press, Berkeley, Calif.

\bibitem[Wright, 1931]{Wright31}
Wright, S. (1931).
\newblock Evolution in mendelian populations.
\newblock {\em Genetics}, 16(2):97.

\end{thebibliography}

\end{document}